\documentclass{article}

\usepackage[paperheight=29.7cm,paperwidth=21cm, hmargin={2.0cm,2.0cm},vmargin={2.3cm,2.0cm}, headheight=14pt]{geometry}

\usepackage[utf8]{inputenc}
\usepackage{dmaths}
\usepackage{color}
\usepackage{enumerate}
\usepackage{wrapfig}

\usepackage{csquotes}
\usepackage{tikz}
\usetikzlibrary{calc}

\usepackage{subcaption}

\newcommand{\lie}{\mathcal{L}}
\newcommand{\F}{\mathcal{F}}

\usepackage{biblatex}
\title{Poisson cohomology of 3D Lie algebras}
\author{Douwe Hoekstra, Florian Zeiser}
\date{\today}

\begin{document}

\maketitle

\begin{abstract}
    We compute the Poisson cohomology associated with several three dimensional Lie algebras. Together with existing results and the classification of three dimensional Lie algebras, this provides the Poisson cohomology of all linear Poisson structures in dimension $3$.
\end{abstract}

\tableofcontents

\section{Introduction}
A Poisson structure on a smooth manifold is a bivector field $\pi \in \mathfrak{X}^2(M):=\Gamma(\wedge^2TM)$ subject to the condition
\begin{equation}\label{eq: differential}
    [\pi,\pi]=0.
\end{equation}
Here $[\cdot,\cdot]$ denotes the Schouten-Nijenhuis bracket, an extension of the Lie bracket to the multivector fields $\mathfrak{X}^{\bullet}(M):=\Gamma(\wedge^{\bullet}TM)$ on $M$, providing the structure of a graded Lie algebra on $\mathfrak{X}^{\bullet}(M)$. Equation \eqref{eq: differential} and the graded Jacobi identity imply  that every Poisson structure induces the differential
\begin{equation*}
    \diff_{\pi}:=[\pi,\cdot ]:\  \mathfrak{X}^{\bullet}(M)\to \mathfrak{X}^{\bullet+1}(M), \qquad \text{ with }\ \diff_{\pi}^2=0.
\end{equation*}
The resulting cohomology is called \textbf{Poisson cohomology}, denoted by $H^{\bullet}(M,\pi)$, which was first introduced by Lichnerowicz in \cite{Lichnerowicz1977}. Poisson cohomology is interesting from an algebraic and a geometric perspective. Algebraically, the wedge product and the Schouten-Nijenhuis bracket descend to cohomology, providing the structure of a Gerstenhaber algebra (\cite{Vaisman1994}):
\begin{equation*}
    (H^{\bullet}(M,\pi),[\cdot,\cdot],\wedge)
\end{equation*}
Geometrically, the cohomology groups encode interesting information about the Poisson manifold $(M,\pi)$, i.e.
\begin{itemize}
    \item $H^{0}(M,\pi)$ is the space of Casimir functions, i.e. smooth functions on $M$ which are constant along the closure of the leaves of $\pi$;
    \item $(H^1(M,\pi),[\cdot,\cdot])$ can be seen as the Lie algebra of the group of outer Poisson automorphism of $(M,\pi)$;
    \item $H^2(M,\pi)$ controls infinitesimal deformations modulo deformations induced by diffeomorphisms, as such it plays a central role in the question of linearization \cite{Conn1985}, \cite{Crainic2011}, \cite{Weinstein1983}, and deformations \cite{FLS},  \cite{KV1}, \cite{Wein87};
    \item $H^3(M,\pi)$ provides obstructions to extend infinitesimal deformations of $(M,\pi)$ to actual deformations;
\end{itemize} 
For a further exposition and more details see for example \cite{Wein98} or \cite{DZ}. It is not yet fully understood to which extend these interpretations are true in general. One reason for that is that there are few explicit computations of Poisson cohomology groups known, due to a lack of general methods for the calculation. Some examples of such calculations for different classes of Poisson structures can be found in  \cite{Ginzburg1996}, \cite{Ginzburg1992}, \cite{Xu92} and for low dimensions in \cite{Gammella}, \cite{Marcut2019}, \cite{Monnier2002}, \cite{Nakanishi1997} and \cite{Roy}.

In this paper we expand the list of examples, as we compute the Poisson cohomology groups associated to various $3$-dimensional Lie algebras, together with their algebraic structures. Given a Lie algebra $(\mathfrak{g},[\cdot,\cdot])$ we obtain a linear Poisson structure $\pi$ on $\mathfrak{g}^*$ by
\begin{equation}\label{eq: linear poisson}
    \pi_\xi(X,Y):= \langle \xi ,[X,Y]\rangle_{\mathfrak{g}^*\times \mathfrak{g}} \qquad \text{ for all } \ \xi\in \mathfrak{g}^*,\, X,Y\in \mathfrak{g}.
\end{equation}
We want to point out that the results of this paper together with the trivial case ($\pi=0$) and the results for the two semisimple $3$-dimensional Lie algebras, i.e. for $\mathfrak{so}(3)$ by Ginzburg and Weinstein in \cite{Ginzburg1992} (more generally for all compact semisimple Lie algebras), and for $\mathfrak{sl}_2(\R)$ by M\u{a}rcu\c{t} and the second author in \cite{Marcut2019}, provide the Poisson cohomology associated to all $3$-dimensional Lie algebras up to isomorphism, by the classification of such Lie algebras. Moreover, the results should be compared to results about the tangential Poisson cohomology associated to $3$-dimensional Lie algebras by Gammella in \cite{Gammella}.

The paper is structured as follows:

In \cref{sec:results} we recall the classification of $3$-dimensional Lie algebras and present our results for the Poisson cohomology groups of the associated linear Poisson structures on their dual.

In \cref{sec:prelim} we fix some notation, give some identities for the Poisson differentials and outline some of the techniques used for the calculations later on.

In \cref{sec:heisenberg} up to \cref{sec:typevii} we prove the results for the various $3$-dimensional Lie algebras. In order to obtain our results we use different techniques, i.e. direct computations (\cref{sec:heisenberg} $\&$ \cref{sec:affine}), averaging over the action of a compact Lie group as in \cite{Ginzburg1992} and classifying $S^1$-invariant multivector fields in $\R^3$ (\cref{sec:euclidean}), and using the splitting into (partially-)formal and flat Poisson cohomology along various Poisson submanifolds, generalizing methods used in \cite{Ginzburg1992}, \cite{Roy} and \cite{Marcut2019} (see \cref{sec:prelim} for more details). In the last case, the formal part is obtained by direct computations and the flat part either using a Poisson diffeomorphism away from the Poisson submanifold which preserves the flat Poisson complex due to mild singularities (\cref{sec:typevi} $\&$ \cref{sec:typeiv}), or via short exact sequences (\cref{sec:typevii}).

\textbf{Acknowledgements} 

We would like to thank Ioan M\u{a}rcu\c{t} for many useful discussions and his supervision of the first authors' Bachelor thesis, where this paper originated from. The second author would like to thank the the Max Planck Institute for Mathematics in Bonn for its hospitality and financial support, during the early stages of this project.

\section{Results}\label{sec:results}
In this section we state the results of the paper and provide some geometric interpretations for them. First we recall the classification of $3$-dimensional Lie algebras in \cref{subsec: classification}. Based on the classification, we give the results of the Poisson cohomology groups for the associated linear Poisson structures on their duals. As mentioned in the introduction, the Poisson cohomology groups have various geometric interpretations in the different degrees. As an example we work these interpretations out in detail for the linear Poisson structure associated with the Heisenberg Lie algebra in \cref{subsec:heisenberg}. For the other Lie algebras we only provide interpretations of classes which are of special importance. 

\subsection{The classification of $3$-dimensional Lie algebras}\label{subsec: classification}
We begin by recalling the classification of $3$-dimensional Lie algebras, originally due to Bianchi ~\autocite{Bia} (also e.g. ~\autocite{Bowers2005}). We distinguish by the dimension of the first derived algebra $\gf^{(1)} = [\gf,\gf]$ of $\gf$:
\begin{enumerate}
    \item zero-dimensional: the abelian Lie algebra 
	\item one-dimensional:
		\begin{enumerate}
		    \item The Heissenberg Lie algebra determined by $[e_1, e_2] = e_3$ 
			\item The direct product of the non-abelian two-dimensional Lie algebra and the one-dimensional Lie algebra given by $[e_1,e_2] = e_1$ 
		\end{enumerate}
	\item two-dimensional:
		\begin{enumerate}
		    \item The Euclidean Lie algebra given by the bracket $[e_1, e_3] = - e_2$ and $[e_1, e_3] = e_1 $
			\item The Lie algebra parametrized by $0 < \tau \leq 1$ (open book-type) and $0 > \tau \ge -1$ (hyperbolic-type) with non-trivial brackets given by $[e_1, e_3] = e_1$ and $[e_2, e_3] = \tau e_2$ 
			\item The Lie algebra with non-trivial brackets $[e_1,e_3] = e_1$ and $[e_2, e_3] = e_1 + e_2$. (semi open book)
			\item The Lie algebras with parameter $0<\tau $ and brackets $[e_1, e_3] = \tau e_1 - e_2$ and $[e_1, e_3] = e_1 + \tau e_2$ (spiral-type).
		\end{enumerate}
	\item three-dimensional: 
		\begin{enumerate}
			\item The semisimple Lie algebra $\sla(2,\R)$
			\item The compact semisimple Lie algebra  $\so(3)$
		\end{enumerate}
\end{enumerate}

We present the results (and the computations) in the order of the classification. To describe the results we identify $\mathfrak{g}^*$ with $\R^3$ and coordinates $(x,y,z)$, using the basis for $\gf$ as in the classification.

\subsubsection*{The abelian Lie algebra}
In the abelian case the Poisson structure $\pi=0$ is trivial and the Poisson cohomology is given by all multivector fields on $\mathfrak{g}^*$.

\subsection{The Heissenberg Lie algebra}\label{subsec:heisenberg}
To describe the results for the Heissenberg Lie algebra we define by {the vector field
\begin{align*}
    R:=\frac{x}{2}\partial_x +\frac{y}{2}\partial_y + z\partial_z,
\end{align*}
on $\R^3$.} Moreover, we denote by $C^{\infty}_0(\R^m)$ the smooth functions which vanish at the origin.
\begin{theorem}\label{thm:heisenberg}
Let $\pi = z\partial_x \wedge \partial_y$ be the linear Poisson structure associated with the Heissenberg Lie algebra. In the different degrees the Poisson cohomology classes are uniquely described by the following representatives:
	\begin{itemize}
	    \setlength\itemsep{0em}
		\item in degree zero the Casimir functions are of the form $f(z)$ for $f\in C^{\infty}(\R)$;
		\item in degree one for $f\in C^\infty(\R)$ and $g \in C^\infty_0(\R^2)$ by
			\begin{equation*}
				{f(z)R} +  [g(x,y),\partial_x\wedge \partial_y];
			\end{equation*}
		\item in degree two for $g_1,g_2 \in C^\infty_0(\R^2)$ by
		\begin{equation*}
			[(g_1(x,y) + zg_2(x,y))\partial_z, \partial_x \wedge \partial_y];
		\end{equation*}
		\item in degree three the classes have for $g \in C^\infty(\R^2)$ unique representatives of the form
		\begin{equation*}
			g(x,y)\partial_x\wedge\partial_y\wedge\partial_z
		\end{equation*}
	\end{itemize}
	Using these set of representatives the algebraic structure is determined by the following relations in $H^\bullet(\R^3,\pi)$:
	\begin{align*}
        \left[[h_0(x,y,z),\partial_x\wedge \partial_y]\right]&=\left[[h(x,y,0),\partial_x\wedge \partial_y]\right]\\
        \left[[h_1(x,y,z)\partial_z,\partial_x\wedge \partial_y]+ h_2(x,y,z)\partial_x\wedge \partial_y\right]&=\left[[h_1(x,y,0)+z(\partial_z h_1(x,y,0) -h_2(x,y,0))\big)\partial_z,\partial_x\wedge \partial_y]\right] \\
        \left[h_3(x,y,z)\partial_x\wedge \partial_y\wedge \partial_z\right] &=\left[ h_3(x,y,0)\partial_x\wedge \partial_y\wedge \partial_z \right]
    \end{align*}
    for any $h_i\in C^{\infty}(\R^3)$, $i=0,1,2,3$ which makes the multivector fields Poisson.
\end{theorem}

\subsubsection*{The geometric interpretation}

 \begin{wrapfigure}{r}{4cm}
  \vspace{-50pt}
 \includegraphics[scale=0.3]{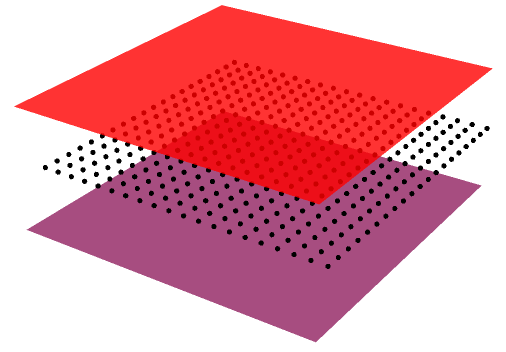}
 \caption{The foliation of $\pi=z\partial_x \wedge \partial_y$}
 \vspace{-30pt}
 \end{wrapfigure}

We give a geometric interpretation of these results, exemplary for all other results. For the Lie algebra $\mathfrak{sl}_2(\R)$ there is a similar interpretation of the results in \cite{Marcut2019}.

In order to understand these results geometrically, let us first understand the leaf space of the symplectic foliation induced by $\pi$. The leaves of $\pi$ are obtained by small flows of the Hamiltonian vector fields, i.e.
\begin{equation*}
    X_g:= \pi^\sharp(\diff g) \qquad \text{ where }\ g\in C^\infty (\R^3).
\end{equation*}
Hence we have two different types of leaves, the planes parallel to the $x\text{-}y$ plane for $z\in \R\backslash \{0\}$ and for $z=0$ the points of the $x\text{-}y$ plane. As a consequence, the leaf space is topologically a real line with a plane of origins.

Casimir functions are precisely the smooth functions on $\R^3$ which are constant along these small flows and accumulation points thereof (which are contained in the leaves here). This is reflected in the result for $H^0(\R^3,\pi)$.\vspace*{5pt}

\textbf{The Lie algebra $H^1(\R^3,\pi)$:}\\
For the interpretation of $H^1(\R^3,\pi)$ we denote the Poisson vector fields by
\begin{equation*}
    \mathfrak{poiss}:=\{X\in \mathfrak{X}^1(\R^3)|\diff_\pi X=0\}.
\end{equation*}
Note that the set of Hamiltonian vector fields 
\begin{equation*}
    \mathfrak{ham}:=\{X_g|g\in C^\infty(\R^3)\}\subset \mathfrak{poiss}
\end{equation*}
forms a Lie ideal in the Lie algebra $(\mathfrak{poiss},[\cdot,\cdot] )$ and the first Poisson cohomology is the Lie algebra on the quotient, i.e.
\begin{equation*}
    H^1(\R^3,\pi)= \mathfrak{poiss}/\mathfrak{ham}.
\end{equation*}
The result for the bracket implies that we have a $2$-term filtration of $\mathfrak{poiss}$
\begin{equation*}
    \mathfrak{ham} \unlhd (\mathfrak{ham} \rtimes \mathfrak{g}_{x\text{-}y})\unlhd \big((\mathfrak{ham} \rtimes \mathfrak{g}_{x\text{-}y})\rtimes \mathfrak{g}_z\big)=\mathfrak{poiss}
\end{equation*}
where the Lie algebras $\mathfrak{g}_{x\text{-}y}$ and $\mathfrak{g}_z$ are given by
\begin{equation*}
    \mathfrak{g}_{x\text{-}y}:= \{\left[[g(x,y),\partial_x\wedge\partial_y]\right]|g\in C^\infty _0 (\R^2)\} \qquad \text{ and }\qquad \mathfrak{g}_z:= \{ \left[{f(z)R}\right]|f\in C^\infty(\R)\}.
\end{equation*}
\begin{remark}\label{remark: lie algebras nil}
    Note that $\mathfrak{g}_{x\text{-}y}$ is naturally isomorphic to $\mathfrak{symp}(\R^2)$, the Lie algebra of symplectic vector fields on $\R^2$, via the projection onto the $x\text{-}y$-plane and $\mathfrak{g}_z$ is isomorphic to $\mathfrak{X}_0^1(\R)$, the Lie algebra of vector fields on $\R$ vanishing at the origin, via the projection onto the $z$-axis.
\end{remark}
A first interesting question is related to formality of the representatives.
\begin{question}
Is it possible to realize $H^1(\R^3,\pi)$ as a Lie subalgebra of $\mathfrak{poiss}$.
\end{question}\vspace*{5pt}

\textbf{The Poisson diffeomorphism group:}\\
To discuss the groups corresponding to the Lie algebras we denote the group of Poisson-diffeomorphisms by
\begin{equation*}
    \mathrm{Poiss}:=\{\varphi \in \mathrm{Diff}(\R^3)| \varphi^\star (\pi)=\pi\}
\end{equation*}
The associated normal subgroup $\mathrm{Ham}\unlhd \mathrm{Poiss}$ of Hamiltonian-diffeomorphisms is given by diffeomorphisms $\varphi\in \mathrm{Poiss}$ which are connected to the identity by a smooth family of diffeomorphisms $\{\varphi_t\}_{t\in [0,1]}$, generated by a smooth family of Hamiltonian vector fields $\{X_{g_t}\}_{t\in [0,1]}$ associated to $\{g_t\in C^\infty(\R^3)\}_{t\in [0,1]}$.
\begin{remark}
    Very little is known about the smoothness of the groups or if they integrate the Lie algebras. Some results have been obtained in \cite{Marcut21}, \cite{Smi2} and \cite{Smilde1}.
\end{remark}
Note that all non-trivial cohomology classes have representatives transverse to the foliation.
\begin{question}
Is any Poisson-diffeomorphism which preserves the leaves of $\pi$ Hamiltonian?
\end{question}
In correspondence with \cref{remark: lie algebras nil} we define the group $G_{x\text{-}y}\simeq \mathrm{Symp}(\R^2)$ by extending a symplectomorphism $\varphi \in \mathrm{Symp}(\R^2)$ of $\R^2$ to a diffeomorphism $\tilde{\varphi}= \varphi\times \mathrm{id}\in \mathrm{Poiss}$. Hence it is natural to ask:
\begin{question}
Is the subgroup of $\mathrm{Poiss}$ which fixes $\R\backslash \{0\}$ in the leaf space of $\pi$ isomorphic to 
\begin{equation*}
    \mathrm{Ham}\rtimes G_{x\text{-}y}?
\end{equation*}
\end{question}
Similarly we define the groups $G_z^0\simeq \mathrm{Diff}^+_0(\R)$ and $G_z\simeq \mathrm{Diff}_0(\R)$ where $\mathrm{Diff}_0^+(\R)$ and $\mathrm{Diff}_0(\R)$ denote the groups of (orientation-preserving) diffeomorphisms on $\R$ which preserve the origin, respectively. We define $\tilde{\varphi}\in G_z^0$ for every $\varphi \in \mathrm{Diff}_0^+(\R)$ in the following way: every $\varphi\in \mathrm{Diff}_0^+(\R)$ is isotopic to the identity (see for example \cite{Mil}[Section 6, Lemma 2]), hence there exists $\{\varphi_t\}_{t\in [0,1]}$ and $X_t\in \mathfrak{X}^1_0(\R)$ with
\begin{equation*}
    \varphi_t' =X_t\circ \varphi_t \qquad \text{ such that } \qquad \varphi_0 = \mathrm{id} \qquad \text{ and }\qquad \varphi_1=\varphi
\end{equation*}
We define a smooth family of functions $f_t\in C^{\infty}(\R)$ by
\begin{equation*}
    f_t(z):= \iota_{\diff z} \left(\frac{X_t(z)}{z}\right)
\end{equation*}
which is well-defined since $X_t$ vanishes at the origin for all $t\in[0,1]$ (see \cref{lemma: hadamard}). {Hence the vector field $f_t(z)R$ induces the isotopy of Poisson diffeomorphism given by
\begin{align*}
    \tilde{\varphi}_t (x,y,z) = \left(\frac{x}{2}e^{\int_0 ^t f_s(\varphi_s(z))\diff s},\frac{y}{2}e^{\int_0 ^t f_s(\varphi_s(z))\diff s},\varphi_t(z)\right) \qquad \text{ and } \qquad \tilde{\varphi}:=\tilde{\varphi}_1
\end{align*}}
Define $\tau: \R \to \R$ by $\tau (z)=-z$ and note that
\begin{equation*}
    \mathrm{Diff}_0(\R)=\mathrm{Diff}^+_0(\R)\cup \tau \cdot \mathrm{Diff}^+_0(\R). 
\end{equation*}
We can extend $\tau$ to an element $\tilde{\tau}\in \mathrm{Poiss}$ by $\tilde{\tau}(x,y,z):=(\tau(x),y,\tau(z))$ and set
\begin{equation*}
    G_z:= G_{z}^0 \cup \tilde{\tau}\cdot G_z^0.
\end{equation*}
which yields the question:
\begin{question}
Is the group of Poisson diffeomorphism $\mathrm{Poiss}$ isomorphic to
\begin{equation*}
    \left(\mathrm{Ham}\rtimes \mathrm{Symp}(\R^2)\right) \rtimes \mathrm{Diff}_0(\R)
\end{equation*}
where the actions are given via the corresponding identifications above, and similarly for $\mathrm{Poiss}^0$.
\end{question}
Another interesting question which we only mention here is the relation of $\mathrm{Poiss}$ and the Picard group \cite{BF},\cite{BW}.\vspace*{5pt}

\textbf{Linearization and deformations: the groups $H^2(\R^3,\pi)$ and $H^3(\R^3,\pi)$}\\
The second Poisson cohomology group $H^2(\R^3,\pi)$ controls infinitesimal deformations of the Poisson structure $\pi$, as such it has the heuristic interpretation as "tangent space" to the Poisson-moduli space. Let us define
\begin{equation}\label{eq: inf deformation}
    \pi_{g_1,g_2}:= \pi + [(g_1(x,y)+zg_2(x,y))\partial_z ,\partial_x\wedge\partial_y]
\end{equation}
for $g_1,g_2\in C^{\infty}_0(\R^2)$, i.e. we deform the Poisson bivector $\pi$ by a general representative of a class in $H^2(\R^3,\pi)$. A direct computation shows that
\begin{equation}\label{eq: bracket inf deformation}
    [\pi_{g_1,g_2},\pi_{g_1,g_2}]=2(\partial_xg_2\partial_yg_1-\partial_xg_1\partial_yg_2)\partial_x\wedge\partial_y\wedge\partial_z 
\end{equation}
Hence $\pi_{g_1,g_2}$ is Poisson iff $\partial_xg_2\partial_yg_1=\partial_xg_1\partial_yg_2$ and its obstruction to being Poisson is given by the class
\begin{equation*}
    \left[2(\partial_xg_2\partial_yg_1-\partial_xg_1\partial_yg_2)\partial_x\wedge\partial_y\wedge\partial_z\right] \in H^3(\R^3,\pi)
\end{equation*}
One might now wonder if all deformations of $\pi$ are of this form, i.e.
\begin{question}
Is every Poisson structure close to $\pi$ of the form $\pi_{g_1,g_2}$ for some $g_1,g_2\in C^{\infty}_0(\R^2)$ satisfying $\partial_xg_2\partial_yg_1=\partial_xg_1\partial_yg_2$?
\end{question}
A special class of Poisson structures are those which are \emph{unimodular}. On an oriented manifold $M^m$ these are Poisson structures $\pi$ which admit a volume form $\mu\in \Omega^m(M)$ invariant under all Hamiltonian vector fields or equivalently
\begin{equation}\label{eq: unimodular}
    \diff \iota_{\pi }\mu = 0
\end{equation}
The Poisson structure $\pi$ corresponding to the Heissenberg Lie algebra satisfies \eqref{eq: unimodular} with respect to the standard volume form $\mu=\diff x\wedge\diff y\wedge\diff z$  on $\R^3$, and hence it is unimodular. Let us consider deformations of the form \eqref{eq: inf deformation} which preserve the Poisson submanifold of the zero dimensional symplectic leaves, i.e. the $x\text{-}y$-plane, i.e. deformations of the form $\pi_{0,g_2}$. Note that such deformations are unobstructed due to \eqref{eq: bracket inf deformation}. Moreover, by a direct calculation one can verify that the Poisson structures $\pi_{0,g_2}$ are unimodular iff $g_2=0$.
\begin{question}\label{question: lin nil}
Is any unimodular Poisson structure $\tilde{\pi}$ whose first jet $j^1\tilde{\pi}$ is isomorphic to $j^1\pi$ along the $x\text{-}y$-plane, locally isomorphic to $\pi$?
\end{question}
Here the first jet map $j^1$ is just given by the Taylor series of the coefficient functions up to order one in $z$.

\subsection{The direct product Lie algebra}\label{subsec:dp}

For the linear Poisson structure associated with  $\mathfrak{aff}(1,\R) \times \R $ we obtain the following result.
\begin{theorem}\label{thm:affine3}
Let $\pi = x \partial_x \wedge \partial_y$ on $\R^3$. The associated Poisson cohomology groups are uniquely described by:
\begin{itemize}
    \setlength\itemsep{0em}
	\item the Casimir functions are of the form $f(z)$ for $f\in C^{\infty}(\R)$;
	\item in degree $1$ we have a free module over the Casimirs with two generators:
	\begin{align*}
	    \langle \partial_y,\partial_z\rangle_{H^0(\R^3,\pi)};
	\end{align*}
	\item in degree $2$ we have a free module over the Casimirs with one generator:
	\begin{align*}
	    \langle \partial_y\wedge\partial_z\rangle_{H^0(\R^3,\pi)};
	\end{align*}
	\item the third Poisson cohomology group is trivial.
\end{itemize}
The wedge product and the induced bracket preserve the representatives, which yields the algebraic structure.
\end{theorem}

 \begin{wrapfigure}{r}{4cm}
  \vspace{-150pt}
 \includegraphics[scale=0.3]{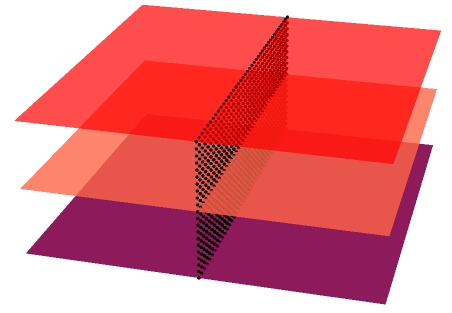}
 \caption{The foliation of $\pi=x\partial_x \wedge \partial_y$}
 \vspace{-100pt}
 \end{wrapfigure}

Similar as for the Heissenberg Lie algebra one can give geometric interpretations of these results. However, we will not do this, but rather we want to point out two observations:
\begin{enumerate}
    \item The modular class and foliation: any oriented Poisson manifold $(M,\pi)$ with volume form $\mu\in \Omega^m(M)$ admits a unique class $\textrm{mod}(M,\pi)\in H^1(M,\pi)$ called the \emph{modular class} defined as follows: the map
    \begin{align*}
        g\mu \mapsto  \Lie_{X_g}\mu \qquad \text{ for any }\qquad g\in C^{\infty}(M)
    \end{align*}
    is a derivation and hence it is given by a vector field $X_{\mu} \in \mathfrak{X}^1(M)$, the \emph{modular vector field} of $(M,\pi,\mu)$. The vector field $X_{\mu}$ is Poisson, i.e. $\diff_{\pi}X_{\mu}=0$ and the class $\textrm{mod}(M,\pi)=[X_\mu]\in H^1(M,\pi)$ is independent of the chosen volume form $\mu$. It is well-known that if $(M,\pi)$ is unimodular iff there exists a volume form $\mu$ such that $X_\mu=0$ (see e.g.\, \cite{Laurent2013}[chapter 4]). In this example the modular class is given by $-\partial_y$. Note that as for the Heissenberg Lie algebra, we have a plane of $0$-dimensional symplectic leaves. However, the values of a general Casimir function is not constant on this plane, but rather only on the flow lines of the modular vector field $\partial_y$, i.e. the leaves of the modular foliation $\tilde{\F}$ given by the distribution
    \begin{equation*}
        T\tilde{\F} = \mathrm{Im}\, (\pi^{\sharp}) \oplus \partial_y
    \end{equation*}
    \item A K\"unneth type formula for Poisson cohomology: the Lie algebra is the direct product of the non-abelian $2$-dimensional Lie algebra $\mathfrak{aff}(1,\R)$ and the $1$-dimensional Lie algebra $\R$. Denote the corresponding linear Poisson structures by $\pi _{\mathfrak{aff}(1,\R)}$ and $0$, respectively. Then $(\R^3,\pi)=(\R^2\times \R, \pi_{\mathfrak{aff}(1,\R)}\oplus 0)$ and we obtain
    \begin{equation*}
        H^{\bullet}(\R^3,\pi) \simeq \bigoplus_{i+j=\bullet} H^i(\R^2,\pi_{\mathfrak{aff}(1,\R)}) \otimes H^j(\R,0).
    \end{equation*}
    A K\"unneth type formula for Poisson cohomology is only known for some cases, even if one Poisson structure is trivial (e.g.\, ~\autocite[Example 9.35]{CFM}).
\end{enumerate}

\subsection{The Euclidean Lie algebra}
Next we look at the Poisson structure associated with the Euclidean Lie algebra $\mathfrak{e}(2)$. Let us define
\begin{align*}
    E:= x\partial_x +y\partial_y \qquad \text{ and } T:= -y\partial_x +x\partial_y.
\end{align*}
The Poisson cohomology groups are given in the following theorem.
\begin{theorem}\label{thm:cohomology_euclidean}
	Let $\pi = T\wedge \partial_z$ be the Poisson structure on $\R^3$ associated with $\mathfrak{e}(2)$. The Poisson cohomology of $\pi$ is given as follows:
	\begin{itemize}
    \setlength\itemsep{0em}
	\item the Casimir functions are of the form $f(x^2+y^2)$ for $f\in C^{\infty}([0,\infty))$;
	\item in degree $1$ we have a free module over the Casimirs with two generators:
	\begin{align*}
	    \langle E,\partial_z\rangle_{H^0(\R^3,\pi)};
	\end{align*}
	\item in degree $2$ the cohomology classes are uniquely represented by
	\begin{align*}
	    f(x^2+y^2)E\wedge \partial_z +g(z)\partial_x\wedge \partial_y;
	\end{align*}
	\item in degree $3$ the cohomology classes are uniquely represented by
	\begin{align*}
	     g(z)\partial_x\wedge \partial_y\wedge \partial_z;
	\end{align*}
\end{itemize}
The algebraic structure is described by the following relations in cohomology:
\begin{align*}
    \left[h(x^2+y^2,z)\partial_x\wedge \partial_y\right]=\left[h(0,z)\partial_x\wedge \partial_y\right] \qquad \text{ and }\qquad \left[h(x^2+y^2,z)\partial_x\wedge \partial_y\wedge \partial_z\right]=\left[h(0,z)\partial_x\wedge \partial_y\wedge\partial_z\right]
\end{align*}
for any $h\in C^{\infty}([0,\infty)\times \R)$ for which the given multivector fields are Poisson.
\end{theorem}

 \begin{wrapfigure}{r}{4cm}
  \vspace{-220pt}
 \includegraphics[scale=0.3]{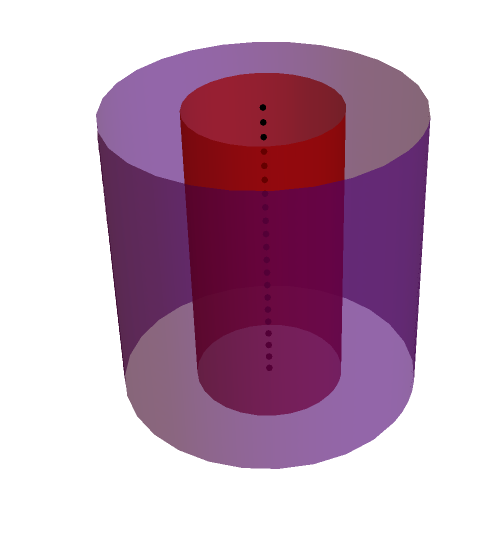}
 \caption{The foliation of $\pi=T \wedge \partial_z$}
 \vspace{-100pt}
 \end{wrapfigure}

For this example we want to take a closer look at its deformations.\vspace*{5pt}

\textbf{Deformations}\\
 Infinitesimal deformations of $\pi$ are governed by $H^2(\R^3,\pi)$. Hence let us define
\begin{equation*}
    \pi_{f,g}:= \pi + f(x^2+y^2)E\wedge\partial_z +g(z)\partial_x\wedge\partial_y
\end{equation*}
for $f\in C^\infty ([0,\infty))$ and $g\in C^{\infty}_0(\R)$. By a direct computation we obtain
\begin{equation*}
    [\pi_{f,g},\pi_{f,g}]=4g(z)(f(x^2+y^2)+(x^2+y^2)\partial f(x^2+y^2))\partial_x\wedge\partial_y\wedge\partial_z 
\end{equation*}
\cref{thm:cohomology_euclidean} implies that in cohomology the corresponding class is represented by 
\begin{equation*}
    \left[[\pi_{f,g},\pi_{f,g}]\right]=\left[4f(0)g(z)\partial_x\wedge\partial_y\wedge\partial_z\right] \in H^3(\R^3,\pi)
\end{equation*}
By formal deformation theory (see \cite{Laurent2013}[chapter 13]) the vanishing of this cohomology class, i.e. $f(0)=0$ or $g(z)=0$ implies that we can find formal deformations of $\pi$, i.e. in terms of power series, for such $f$ and $g$. It would be interesting to understand whether this is also true in the smooth category.
\begin{question}
Does there exist a Poisson deformation $\tilde{\pi}_{f,g}$ of $\pi$ for every $f\in C^\infty ([0,\infty))$ and $g\in C^{\infty}_0(\R)$ with $f(0)=0$ or $g(z)=0$?
\end{question}
Let us look at the deformations corresponding to $f$ and $g$ respectively, separately. Consider first the deformations given by $f=0$. We observe that the Poisson structures $\pi_{0,g}$ are unimodular for any choice of $g$. As an example we may choose $g\in \R$. Then the leaves of $\pi_{0,g}$ are given by the level sets of the functions
\begin{equation*}
    x^2+y^2-gz. 
\end{equation*}
The foliations of the Poisson structures $\pi_{f,0}$ on the other hand yield spiraling planes toward the $z$-axis as leaves, similar to the deformations of $\mathfrak{sl}_2(\R)$ (see \cite{Weinstein1983}[section 6]). Note that the Poisson structures of this form are not unimodular, which motivates the question similar to \cref{question: lin nil}:
\begin{question}
Is any unimodular Poisson structure $\tilde{\pi}$ whose first jet $j^1\tilde{\pi}$ is isomorphic to $j^1\pi$ along the $z$-axis, locally isomorphic to $\pi$?
\end{question}

\subsection{Open book and hyperbolic-type Lie algebras}
Now we consider the Lie algebras given by the non-trivial brackets
\[ [e_1, e_3] = e_1\qquad \text{ and }\qquad [e_2, e_3] = \tau e_2 \qquad \text{ for }\ \tau\in [-1,1]\backslash \{0\}\]
We define the vector fields
\begin{align*}
    E_{\tau}:= x\partial_x +\tau y\partial_y
\end{align*}
with $E_1=E$. We distinguish five different cases for $\tau$, three for the open book-type Lie algebras ($0<\tau\le 1$) and two for the hyperbolic-type Lie algebras ($-1\le \tau <0$). The case $\tau=0$ was discussed in section \ref{subsec:dp}.
\subsubsection*{The open book-type Lie algebras}
For $0<\tau\le 1$ we distinguish three cases
\begin{equation*}
    {\color{blue}\tau=1}, \qquad {\color{purple}\tau=\frac{1}{n} \quad \text{ for }2\le n\in \N},\qquad \text{ and }\qquad {\color{orange} \tau \in (0,1]\backslash \left\{\frac{1}{n}\right\}_{n\in \N}}. 
\end{equation*}
which we indicate by the given color code.
\begin{theorem}\label{thm:open_book t1}
    The Poisson cohomology of $\pi_{\tau}=E_{\tau}\wedge \partial_z$ for $\tau\in (0,1]$ on $\R^3$ is described by:
    \begin{itemize}
        \item the Casimir functions are given by elements in $\R$.
        \item in degree $1$ we have a free module over the Casimir functions with generators:
        \begin{align*}
            {\color{blue}\langle y\partial_x ,x\partial_y,E_{-1},\partial_z\rangle}, \qquad 
            {\color{purple}\langle y^n\partial_x ,E_{-1},\partial_z\rangle}\qquad\text{ and }\qquad
            {\color{orange}\langle E_{-1},\partial_z\rangle}\qquad \text{ respectively;}
        \end{align*}
        \item in degree $2$ we have a free module over the Casimir functions with generators:
        \begin{align*}
            {\color{blue}\langle y\partial_x\wedge \partial_z,x\partial_y\wedge \partial_z,E_{-1}\wedge \partial_z\rangle}, \qquad 
            {\color{purple}\langle y^n\partial_x\wedge \partial_z,E_{-1}\wedge \partial_z\rangle} \qquad \text{ and }\qquad
            {\color{orange}\langle E_{-1}\wedge \partial_z\rangle} \qquad \text{ respectively;}
        \end{align*}
        \item and the groups $H^3(\R^3,\pi)$ are trivial.
    \end{itemize}
    The algebraic structures is determined by the relations:
    \begin{align*}
        {\color{blue}\left[E\wedge\partial_z\right]=\left[0\right] \in H^2(\R^3,\pi_1),\quad \left[p(x,y)\partial_x\wedge\partial_y\right] =\left[0\right]\in H^2(\R^3,\pi_1)}\quad \text{ and } \quad {\color{purple}\left[y^{n+1}\partial_x\wedge\partial_y\right] =\left[0\right]\in H^2(\R^3,\pi_{\frac{1}{n}})}
    \end{align*}
    for any polynomial $p$ homogeneous of degree $2$. In particular, we note that for $\tau \in (0,1]\backslash \{\frac{1}{n}\}_{n\in \N}$ the wedge product and the induced bracket preserve the representatives.
\end{theorem}

     \begin{wrapfigure}{r}{4cm}
  \vspace{-30pt}
 \includegraphics[scale=0.5]{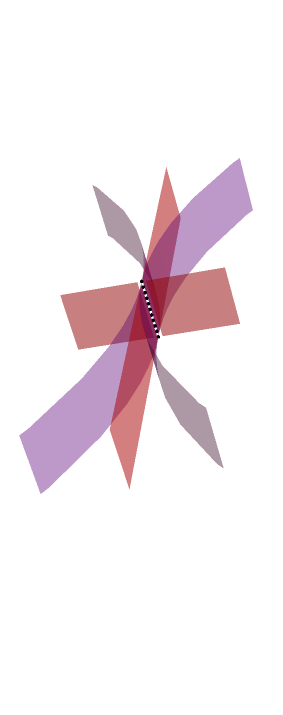}
 \vspace{-60pt}
 \caption{The foliation of $\pi_{\frac{1}{2}}=E_{
\frac{1}{2}} \wedge \partial_z$}
 \vspace{-100pt}
 \end{wrapfigure}

Let us point out two observations: 
\begin{enumerate}
    \item The Lie algebra $H^1(\R^3,\pi_\tau)$ is in the different cases given by:
    \begin{itemize}
        \item for $\tau=1$: the three Poisson vector fields $y\partial_x$, $x\partial_y$ and $E_{-1}$ generate a Lie algebra isomorphic to $\mathfrak{sl}_2(\R)$ and all of them commute with $\partial_z$. Hence we have an isomorphism of Lie algebras:
        \begin{equation*}
            H^1(\R^3,\pi_1)\simeq \mathfrak{sl}_2(\R)\oplus \R
        \end{equation*}
        \item for $\tau=\frac{1}{n}$ with $2\le n\in \N$: the vector fields $y^n\partial_x$ and $E_{-1}$ generate a Lie algebra isomorphic to $\mathfrak{aff}_1(\R)$ and both commute with $\partial_z$. Therefore we obtain in this case an Lie algebra isomorphism:
        \begin{equation*}
            H^1(\R^3,\pi_{\frac{1}{n}})\simeq \mathfrak{aff}_1(\R)\oplus \R
        \end{equation*}
        \item for $\tau \in (0,1]\backslash \{\frac{1}{n}\}_{n\in \N}$ we obtain the Lie algebra isomorphism
        \begin{equation*}
            H^1(\R^3,\pi_{\tau})\simeq \R\oplus \R
        \end{equation*}
    \end{itemize}
    \item The linearization problem: the classes given by representatives of the form
    \begin{equation*}
        y^n\partial_x\wedge \partial_z
    \end{equation*}
    for $2\le n\in \N$ are precisely those which provide obstructions to the linearization problem for Poisson structures. Any Poisson structure $\tilde{\pi}$ with $j^1\tilde{\pi}=j^1\pi_\tau$ for $\tau \in (0,1]\backslash\{\frac{1}{n}\}_{2\le n\in \N}$ is locally around the origin isomorphic to $\pi_\tau$ (see \cite{Duf} or \cite{DZ}[Proposition 4.2.2]). 
\end{enumerate}
\newpage

\subsubsection*{The hyperbolic-type Lie algebras}

     \begin{wrapfigure}{r}{4cm}
  \vspace{-120pt}
 \includegraphics[scale=0.4]{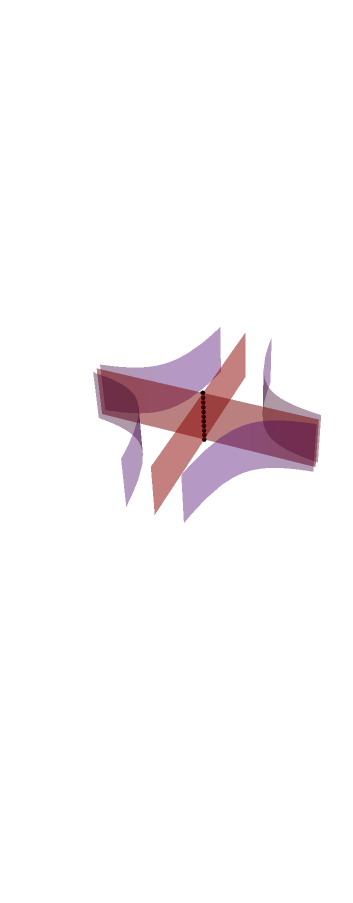}
 \vspace{-110pt}
 \caption{The foliation of $\pi_{-\frac{1}{2}}=E_{
-\frac{1}{2}} \wedge \partial_z$}
 \vspace{-60pt}
 \end{wrapfigure}

For $-1\le \tau<0$ we first consider the two cases
\begin{equation*}
    {\color{brown}\tau=-\frac{p}{q}\quad \text{ with }\ p\le q\in \N\  \text{ relatively prime }}\qquad \text{ and } \qquad {\color{violet}\tau\in [-1,0)\backslash \Q}.
\end{equation*}
To describe the results we set
\begin{equation*}
    C^{\infty}_{0_f}([0,\infty)):=\{\ g\in C^\infty ([0,\infty))\ |\  j^\infty_0g =0\ \}
\end{equation*}
where $j^\infty_0$ denotes the infinite jet at the origin, i.e. the Taylor series of $g$.
\begin{theorem}\label{thm:hyperbolic}
    The Poisson cohomology of $\pi_\tau =E_{\tau}\wedge \partial_z$ on $\R^3$ for $\tau \in [-1,0)$ is given as follows
    \begin{itemize}
        \item the Casimir functions are of the form $f(x,y)$ given by
        \begin{align*}
            {\color{brown}f(x,y):=g(x^py^q)+\begin{cases}
            g_1(x^py^q) &\text{ if } \ 0<x,y\\
            g_2(-x^py^q) &\text{ if } \ x>0>y \ \text{ and } p \text{ odd} \\
            g_2(-x^py^q) &\text{ if } \ y>0>x \ \text{ and } p \text{ even}\\
            0 &\text{ else;}
            \end{cases}
            \quad \text{ for } \ g\in C^{\infty}(\R) \text{ and } g_1,g_2\in C^{\infty}_{0_f}([0,\infty))}
        \end{align*}
        and
        \begin{align*}
            {\color{violet}f(x,y,\tau):=c+\begin{cases}
            g_{ij}(|x|^{-\tau}|y|) &\text{ if } \ 0<(-1)^ix,(-1)^jy\\
            0 & \text{ else,}
            \end{cases}\quad \text{ where } g_{ij}\in C^{\infty}_{0_f}([0,\infty)) \text{ for } i,j=0,1.} 
        \end{align*}
        \item in degree $1$ we have non-trivial classes uniquely represented by:
        \begin{align*}
            {\color{brown}g E+ f(x,y)  \partial_z}\qquad \text{ and } \qquad {\color{violet}f(x,y)  \partial_z} \qquad \text{ for } \ f\in H^0(\R^3,\pi_\tau), g\in \R[[x^py^q]];
        \end{align*}
        \item the group $H^2(\R^3,\pi_\tau)$ has unique representatives of the form
        \begin{equation*}
            {\color{brown} g E\wedge \partial_z} \qquad \text{ and }\qquad {\color{violet} \{0\}} \qquad \text{ for } \ g\in \R[[x^py^q]];
        \end{equation*}
        \item and the groups $H^3(\R^3,\pi_\tau)$ are trivial.
    \end{itemize}
\end{theorem}

As for the open book-type Lie algebras, the representatives of the non-trivial classes in $H^2(\R^3,\pi_\tau)$ are precisely those which yield obstructions to the linearization problem for such Poisson structures.

\subsection{The semi open book-type, spiral-type and the semisimple Lie algebras}
For the Lie algebra of the semi open book-type in the classification we obtain the following result.

 \begin{wrapfigure}{r}{4cm}
  \vspace{-180pt}
 \includegraphics[scale=0.35]{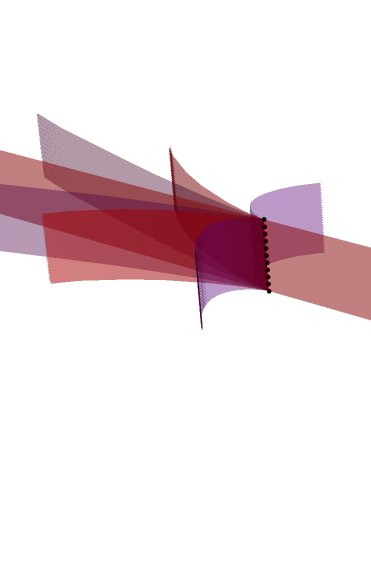}
 \vspace{-65pt}
 \caption{The foliation of $\pi=(E+x\partial_y) \wedge \partial_z$}
 \vspace{35pt}
 \end{wrapfigure}

\begin{theorem}\label{thm: type iv}
The Poisson cohomology groups $H^{\bullet}(\R^3,\pi)$ of the Poisson structure $\pi= (E+x\partial_y)\wedge\partial_z$ are uniquely characterized by the following representatives:
\begin{itemize}
    \item Casimir functions are given by elements in $\R$;
    \item in degree $1$ we have a free module over the Casimir functions with generators
    \[ \langle x\partial_y, \partial_z\rangle;\]
    \item in degree $2$ we have a free module with generator
    \[ \langle y\partial_{x}\wedge\partial_z \rangle; \]
    \item and the third Poisson cohomology group is trivial.
\end{itemize}
In this case all the algebraic structure is trivial with the exception of module structure over the Casimir functions.
\end{theorem}

Poisson structures associated to Lie algebras of the spiral-type have the following Poisson cohomology:
\begin{theorem}\label{thm:typevii}
The Poisson cohomology groups $H^{\bullet}(\R^3,\pi_{\tau})$ for $0<\tau$ with $\pi_\tau = (\tau E +T)\wedge\partial_z$ are uniquely characterized by the following representatives in the different degrees:
\begin{itemize}
    \item Casimir functions are given by elements in $\R$;
    \item in degree $1$ we have a free module over the Casimir functions with generators
    \[ \langle E, \partial_z\rangle;\]
    \item in degree $2$ we have a free module with generator
    \[ \langle E\wedge\partial_z \rangle; \]
    \item and the third Poisson cohomology group is trivial.
\end{itemize}
Moreover, it's algebraic structure is given by the operations on the representatives.
\end{theorem}

 \begin{wrapfigure}{r}{4cm}
  \vspace{-130pt}
 \includegraphics[scale=0.3]{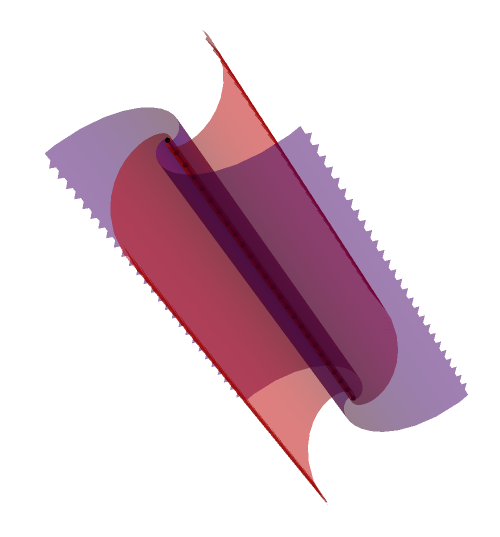}
 \vspace{-20pt}
 \caption{The foliation of $\pi_1=(E+T) \wedge \partial_z$}
 \vspace{0pt}
 \end{wrapfigure}

Finally, there are the two semisimple Lie algebras $\mathfrak{sl}_2(\R)$ and $\mathfrak{so}(3)$. For the linear Poisson structure associated with $\mathfrak{so}(3)$ the result is obtained using a more general result by Ginzburg and Weinstein in \cite{Ginzburg1992} for any compact semisimple Lie algebra. The $\mathfrak{so}(3)$ their result implies:
\begin{theorem}[Ginzburg-Weinstein \cite{Ginzburg1992}]
    The Poisson cohomology of the linear Poisson structure $\pi$ associated with $\mathfrak{so}(3)$ is characterized by
    \begin{itemize}
        \item the Casimir functions are of the form $f(x^2+y^2+z^2)$ for $f\in C^{\infty}([0,\infty))$;
        \item the groups $H^1(\mathfrak{so}(3)^*,\pi)$ and $H^2(\mathfrak{so}(3)^*,\pi)$ are trivial;
        \item $H^3(\mathfrak{so}(3)^*,\pi)$ is a free module over $H^0(\mathfrak{so}(3)^*,\pi)$ generated by the standard $3$-field on $\mathfrak{so}(3)^*$ induced by the Killing form.
    \end{itemize}
\end{theorem}
 For the Poisson cohomology of the linear Poisson structure associated with $\mathfrak{sl}_2(\R)$ we refer the reader to \cite{Marcut2019}.

\section{Preliminaries}\label{sec:prelim}
In this we introduce some general notations, tools and computations which we will use several times throughout the paper. In \cref{subsec:notation} we identify Poisson cohomology of linear Poisson structures as Lie algebra cohomology with coefficients, we write down a general form of the Poisson differential which we frequently use in our calculations and we recall some useful identities. In \cref{subsec:flatformal} we explain how one can split the calculation of Poisson cohomology into a flat and formal part. Finally, in \cref{subsec:corank1} we explain how Poisson cohomology of a corank 1 symplectic foliation can be calculated. 

explain the our general approach for explicit calculations

\subsubsection*{The general plan for the computations}
Our general procedure for the explicit computations of (formal) Poisson cohomology in degree $0\le i$ is as follows:
    \begin{enumerate}
        \item We choose a complement to the cocycles in degree $i-1$ and look at its image under $\diff_{\pi}$ (void for $i=0$). This gives a description of the coboundaries in degree $i$ and allows us to make certain restrictions on representatives of cocyles in this degree. 
        \item Determine the cocycles with respect to the restrictions obtained in $1.$
    \end{enumerate}
    The result will give unique representatives for the cohomology classes.

\subsection{Notation and useful identities}\label{subsec:notation}
For the linear Poisson structure on the dual of a finite dimensional Lie algebra $(\mathfrak{g},[\cdot,\cdot ])$, we identify Poisson cohomology with the Chevalley-Eilenberg cohomology of $\mathfrak{g}$ with coefficients in certain representations. Let $\pi$ be the associated linear Poisson structure on $\mathfrak{g}^*$ given by \eqref{eq: linear poisson} with corresponding Poisson bracket $\{\cdot,\cdot \}$ on $C^{\infty}(\mathfrak{g}^*)$ and $l:\mathfrak{g}\to C^{\infty}(\mathfrak{g}^*)$ the map identifying $\mathfrak{g}$ with $(\mathfrak{g}^*)^*$. With this notation $C^{\infty}(\mathfrak{g}^*)$ becomes a $\mathfrak{g}$-representation, with $X\cdot f:=\{l_X,f\}$. 
Moreover, the Poisson complex of $(\mathfrak{g}^*,\pi)$ is isomorphic to the Chevalley-Eilenberg complex of $\mathfrak{g}$ with coefficients in $C^{\infty}(\mathfrak{g}^*)$ \cite[Prop 7.14]{Laurent2013}
\begin{equation*}
(\mathfrak{X}^{\bullet}(\mathfrak{g}^*),\diff_{\pi})\simeq (\wedge^{\bullet}\mathfrak{g}^*\otimes C^{\infty}(\mathfrak{g}^*),\diff_{EC}).
\end{equation*}
Since we only consider three dimensional lie algebras $\mathfrak{g}$, for the calculation of the Poisson cohomology we usually identify $\mathfrak{g}^*\simeq \R^3$ with coordinates $(x,y,z)$. We consider Poisson cohomology with coefficients in different subr(i)ngs $\mathcal{R}$ of $C^{\infty}(\R^3)$ which are preserved by $\pi$. We denote the corresponding subcomplex of the Poisson complex by $(\mathfrak{X}_{\mathcal{R}}^{\bullet},\diff_{\pi})$, i.e.
\begin{equation}\label{eq: iso cx general}
(\mathfrak{X}^{\bullet}_{\mathcal{R}}(\R^3),\diff_{\pi})\simeq (\wedge^{\bullet}\mathfrak{g}^*\otimes \mathcal{R},\diff_{EC}).
\end{equation}
In general we denote by $g$ an element in $\mathfrak{X}^0_{\mathcal{R}}$, by $X$ an element in $\mathfrak{X}^1_{\mathcal{R}}$, by $W$ an element in $\mathfrak{X}^2_{\mathcal{R}}$. A vector field $X$ and a bivector field $W$ are then of the form
\begin{align*}
    X=X^x\partial_x+X^y\partial_y+X^z\partial_z  \qquad \text{ and }\qquad W=W^x\partial_y \wedge \partial_z+W^y\partial_z \wedge \partial_x+W^z\partial_x \wedge \partial_y 
\end{align*}
with $X^i,W^i\in {\mathcal{R}}$.

Note that all Poisson structures $\pi$ which we consider in this paper are of the form
\begin{align*}
    \pi = X_{\pi}\wedge \partial_z
\end{align*}
for some vector field $X_{\pi}$ on $\R^2$ in $(x,y)$ of homogeneous degree $0$, i.e.\, $m_t^*X_\pi=X_\pi$ where $m_t:(x,y)\mapsto t\cdot (x,y)$ and $t\ne 0$. The only exceptions will be $\mathfrak{heiss}$ and $\mathfrak{aff}(1,\R)\times \R$ where we will exchange the roles of $y$ and $z$ since this notation is more standard.

For such Poisson structures the general Poisson differential in the different degrees is as follows:
\begin{align}
    \diff_{\pi}(g)=&\, (\lie_{\partial_z}g)X_{\pi}-(\lie_{X_{\pi}}g)\partial_z\label{eq: general 0}\\ 
    \diff_{\pi}(X)=&\, = W(\pi)=W^x(\pi)\partial_y \wedge \partial_z+W^y(\pi)\partial_z \wedge \partial_x+W^z(\pi)\partial_x \wedge \partial_y \nonumber \\
    \diff_{\pi}(W)=&\, \Big(X_{\pi}^x \lie_{\partial_z}W^x+X_{\pi}^y \lie_{\partial_z}W^y+(\lie_{\partial_x}X_{\pi}^x+\lie_{\partial_y}X_{\pi}^y-\lie_{X_{\pi}})W^z\Big)\partial_x\wedge\partial_y\wedge\partial_z \label{eq: general 2}
\end{align}
where 
\begin{align}
    W^x(\pi)=&\, \lie_{X_{\pi}}X^y-X^x\lie_{\partial_x}X_{\pi}^y-X^y\lie_{\partial_y}X_{\pi}^y+X_{\pi}^y\lie_{\partial_z}X^z\label{eq: general 1x}\\ 
    W^y(\pi)=&\,-\lie_{X_{\pi}}X^x+X^x\lie_{\partial_x}X_{\pi}^x+X^x\lie_{\partial_x}X_{\pi}^x-X_{\pi}^x\lie_{\partial_z}X^z\label{eq: general 1y}\\
    W^z(\pi)=&\, X_{\pi}^x\lie_{\partial_z}X^y-X_{\pi}^y\lie_{\partial_z}X^x\label{eq: general 1z}
\end{align}
and the Poisson differential is $0$ when applied to $3$-vector fields for dimensional reasons.

A very useful tool for our computations is the following lemma.
\begin{lemma}[Hadamard]\label{lemma: hadamard}
Let $g \in C^\infty(\R^n)$, then:
\begin{align*}
    g(x_1,\dots, x_n)= g_0(x_1,\dots x_{n-1})+x_ng_1(x_1,\dots ,x_n) 
\end{align*}
where 
\begin{align*}
    g_0(x_1,\dots ,x_{n-1})=g(x_1,\dots ,x_{n-1},0), \qquad \text{ and }\qquad  g_1(x_1,\dots ,x_n)=\int_0^1\partial_{x_n}g(x_1,\dots,tx_n)\diff t
\end{align*}
\end{lemma}
As an immediate consequence we obtain:
\begin{corollary}\label{cor: invertible}
On $\R^n$ with coordinates $(x_1,\dots ,x_n)$ the operator
\begin{align*}
    \partial_{x_i}:x_iC^{\infty}(\R^n)\to C^{\infty}(\R^n) 
\end{align*}
is bijective with inverse $\int_0^{x_i}\diff\tilde{x}_i$, for all $i=1,\dots,n$
\end{corollary}

\subsection{Flat and formal Poisson cohomology along Poisson submanifolds}\label{subsec:flatformal}

In some of our computations it is useful to divide the computation of Poisson cohomology into the computation of a flat and formal part along certain Poisson submanifolds. In this section we explain the general idea and define the objects we need later in our calculations.

Let $(M,\pi)$ be a Poisson manifold. Its Poisson cohomology $H^{\bullet}(M,\pi)$ is the cohomology of the chain complex: 
\[(\mathfrak{X}^{\bullet}(M),\diff_{\pi}:=[\pi,\cdot]).\]
For an embedded Poisson submanifold $N \subset M$, let $\mathfrak{X}_{N_f}^{\bullet}(M)$ denote the set of multivector fields that are flat along $N$, i.e. the inverse limit of $\mathcal{I}_N^n\mathfrak{X}^{\bullet}(M)$, where $\mathcal{I}_N$ denotes the ideal of functions which vanish along $N$. Since $\mathfrak{X}_{N_f}^{\bullet}(M)$ is a Lie ideal in $\mathfrak{X}^{\bullet}(M)$, it is also a subcomplex with respect to $\diff_{\pi}$. The cohomology of this complex, denoted $H^{\bullet}_{N_f}(M,\pi)$, will be called the flat Poisson cohomology along $N$. 

An adaptation of Borel's Lemma on the existence of smooth functions with a prescribed Taylor series yields the following identification for the quotient:
\[\mathfrak{X}^{\bullet}(M)/\mathfrak{X}_{N_f}^{\bullet}(M)\simeq \Gamma(\wedge^{\bullet}T_NM\otimes \underset{k\ge 0}{\Pi}S^k\nu^*_N)=:\mathfrak{X}_{N_F}^{\bullet}(M),\]
where $S^k\nu^*_N$ denotes the $k$-th symmetric power in the conormal bundle of $N$. Thus, we obtain a short exact sequence of complexes
 \begin{align}\label{eq: ses general}
  0\to(\mathfrak{X}^{\bullet}_{N_f}(M),\diff_{\pi})\to(\mathfrak{X}^{\bullet}(M),\diff_{\pi})\stackrel{j^{\infty}_N}{\longrightarrow} (\mathfrak{X}_{N_F}^{\bullet}(M),\diff_{j^{\infty}_{N}\pi})\to 0,
 \end{align}
where $j^{\infty}_N$ is the infinite jet map along $N$. The cohomology of the quotient complex, denoted by $H^{\bullet}_{N_F}(M,\pi)$, will be called the formal Poisson cohomology along $N$. The short exact sequence induces a long exact sequence in cohomology: 
\begin{equation}\label{jet}
\ldots \stackrel{j^{\infty}_N}{\to} H^{q-1}_{N_F}(M,\pi)\stackrel{\partial}{\to} 
H^{q}_{N_f}(M,\pi)\to H^{q}(M,\pi)\stackrel{j^{\infty}_N}{\to}H^{q}_{N_F}(M,\pi )\stackrel{\partial}{\to}\ldots. \end{equation}
For us only the following submanifolds will be relevant:
\begin{align*}
    \Rho := \Set{(0,0,z)\in \R^3}[z \in \R],\qquad X:= \Set{(0,y,z)\in \R^3}[y,z \in \R],\qquad Y:= \Set{(x,0,z)\in \R^3}[x,z \in \R] \\ XY:= \Set{(x,y,z)\in \R^3}[x,y,z \in \R :\, x=0\ \text{ or } \ y=0 ].\qquad \qquad \qquad \qquad \quad
\end{align*}
Note that by the identification \eqref{eq: iso cx general}, in order to obtain short exact sequences as in \eqref{eq: ses general} and hence long exact sequences in cohomology, we only need to make sure that these submanifolds are Poisson for a given Poisson structure $\pi$ and that the corresponding sequences for functions are short exact. To ensure the latter we have the following statement.
\begin{proposition}\label{proposition: ses}
The following sequences are short exact sequences:
\begin{align*}
    0\to C_{\Rho _f}^{\infty}(\R^3)\hookrightarrow{} C^{\infty}(\R^3)\xrightarrow{j^{\infty}_R} C^{\infty}(\R)[[x,y]]\to 0\\
    0\to C_{Y_f}^{\infty}(\R^3)\hookrightarrow{} C^{\infty}_{\Rho _f}(\R^3)\xrightarrow{j^{\infty}_Y} C^{\infty}_{U_f}(\R^2)[[y]]\to 0\\
    0\to C_{XY_f}^{\infty}(\R^3)\hookrightarrow{} C^{\infty}_{Y_f}(\R^3)\xrightarrow{j^{\infty}_{X}} C^{\infty}_{U_f}(\R^2)[[x]]\to 0
\end{align*}
where 
\begin{align*}
    C_{Y_f}^{\infty}(\R^3):=\Set{f\in C^{\infty}_{\Rho _f}(\R^3)}[j^{\infty}_Yf=0]\qquad C_{XY_f}^{\infty}(\R^3):=\Set{f\in C^{\infty}_{Y_f}(\R^3)}[j^{\infty}_Xf=0]\\
    C^{\infty}_{U_f}(\R^2):=\Set{f\in C^{\infty}(\R^2)}[j^{\infty}_{U}f=0, \  U= \Set{(0,v)\in \R^2}[v \in \R]]\qquad \qquad
\end{align*}
We denote the multivector fields associated to the corresponding rng of functions by $\mathfrak{X}^{\bullet}_{\Rho _f}(\R^3)$, $\mathfrak{X}^{\bullet}_{\Rho _F}(\R^3)$, $\mathfrak{X}^{\bullet}_{Y_f}(\R^3)$, $\mathfrak{X}^{\bullet}_{Y_F}(\R^3)$, $\mathfrak{X}^{\bullet}_{XY_f}(\R^3)$ and $\mathfrak{X}^{\bullet}_{XY_F}(\R^3)$, respectively.
\end{proposition}

The proof of this proposition is based on the following lemma.
\begin{lemma}[Borel]\label{lemma: Borel}
	Let $f \in C^\infty(\R^{k + m})$, then the map $j^\infty:C^{\infty}(\R^{k+m})\to C^{\infty}(\R^k)[[\R^m]]$ given by
	\begin{equation*}
	    j^\infty_{\R^m}(f) = \sum_{\alpha \in \N_0^m} D^\alpha_{y}f(x,0)y_{1}^{\alpha_1}\cdots y_{m}^{\alpha_m}, \qquad \text{ where }\  D_y^{\alpha}:=\frac{1}{\alpha_1! \dots \alpha_1! }\partial_{y_1}^{\alpha_1}\dots \partial_{y_m}^{\alpha_m}
	\end{equation*}
	is surjective.
\end{lemma}
\begin{proof}
	This proof has been adapted from the proof of theorem I.1.3 in~\cite{Moerdijk1991}.
	Let $g \in C^\infty(\R^k)[[y_1,\ldots,y_m]]$, i.e.
	\begin{equation*}
		g = \sum_{\alpha \in \N_0^m} g_\alpha(x_1,\ldots,x_n) y_1^{\alpha_1}\cdots y_m^{\alpha_m}.
	\end{equation*}

	Let $(U_i)_{i \in I}$ be an open cover of $\R^k$ such that each $\overline{U}_i$ is compact. Let $\phi: \R^m \rightarrow \R$ be a bump function with $\phi|_{\overline{B}_1(0)} = 1$ and $\supp \phi \subseteq B_2(0)$. We first show that there exist constants $c_\alpha > 1$ such that 
	\begin{equation}\label{eq:guessed_function_derivs}
		\sum_{\alpha \in \N_0^m} D^\beta_{(x,y)}\left(g_\alpha(x) \phi(c_\alpha \cdot y) y^\alpha\right)
	\end{equation}
	is uniformly convergent on $\overline{U}_i \times \R^m$ for all $\beta \in \N_0^{k + m}$. 

    For all $\alpha \in \N_0^m$, set $\phi_\alpha(y) = \phi(y)y_1^{\alpha_1}\cdots y_m^{\alpha_m}$, then we have 
	\begin{equation*}
		g_\alpha(x) \phi(c_\alpha y) y^{\alpha} = \left(\frac{1}{c_\alpha}\right)^{|\alpha|}g_\alpha(x)\phi_\alpha(c_\alpha y).
	\end{equation*}
	Since $\phi$ has compact support, the same holds for $\phi_\alpha$ and hence
	\begin{equation*}
		M_\alpha = \max_{|\beta|<|\alpha|} \sup_{x \in \overline{U}_i, y \in \R^m}\left|D^\beta_{(x,y)}\left(g_\alpha(x)\phi_\alpha(y)\right)\right| < \infty.
	\end{equation*}
	For $|\beta| < |\alpha|$ and $(x,y)\in \overline{U}_i \times \R^m$ we obtain the estimate
	\begin{align*}
	    \left|D^\beta_{(x,y)}\left(g_\alpha(x) \phi(c_\alpha y) y^{\alpha}\right)\right| &= \left(\frac{1}{c_\alpha}\right)^{|\alpha|} \left|D^\beta_{(x,y)}(g_\alpha(x)\psi_\alpha(c_\alpha y))\right| \leq \left(\frac{1}{c_\alpha}\right)^{|\alpha|} c_\alpha^{|\beta|} M_\alpha < \frac{M_\alpha}{c_\alpha}.
	\end{align*}
	Let $c_\alpha = 2^{|\alpha|}M_\alpha$, then for each $\beta \in \N_0^{n + m}$, the sum in \cref{eq:guessed_function_derivs} is dominated by $\sum_{\alpha \in \N_0^m} 2^{-|\alpha|}$ for $|\alpha| > |\beta|$, hence it converges uniformly on $\overline{U}_i \times \R^m$ for all $i \in I$. Therefore we have for any $\alpha \in \N^m$ that
	\begin{align*}
	    D^\alpha g|_{U_i}=g_{\alpha}|_{U_i}.
	\end{align*}
    Let $(\rho_i, V_i)_{i \in I}$ be a partition of unity subordinate to this open cover. Then the function $f\in C^{\infty}(\R^{k+m})$ defined by
	\begin{equation*}
		f(x,y) := \sum_{i \in I} \rho_i(x)g(x,y)
	\end{equation*}
	satisfies the desired properties.
\end{proof}

\begin{proof}[Proof of \cref{proposition: ses}]
We note that it suffices to show that the infinite jet maps are surjective. For this we may replace the function $g$ in the previous proof by an element in any of the right hand spaces in the sequences.
\end{proof}

\subsection{Corank 1 Poisson structures}\label{subsec:corank1}
We first briefly recall the necessary theory for corank one symplectic foliation. For details see for example \cite{Torres15}.

Note that for a regular Poisson manifold $(M,\pi)$ of corank one with symplectic foliation $(\mathcal{F}=\pi^{\sharp}(T^*M),\omega)$, where $\omega\in \Omega^2(\F)$ is the leafwise symplectic structure, 
The Poisson complex of $\pi$ fits into a short exact sequence:  
 \begin{align}\label{ses poisson}
  0\to (\Omega^{\bullet}(\mathcal{F}),\diff _{\mathcal{F}})\xrightarrow{j} (\mathfrak{X}^{\bullet}(M),\diff _{\pi})\xrightarrow{p} (\Omega^{\bullet-1}(\mathcal{F},\nu),\diff _{\nabla})\to 0.
 \end{align}
 Here $\diff_{\nabla}$ is the differential induces by the Bott connection on the normal bundle $\nu=TM/\F$. The map $j$ is obtained by pulling back Lie algebroid forms via the Lie algebroid map $\pi^{\sharp}:T^*M\to \F$ and the map $p$ is obtained by using the isomorphism
\[(-\omega^{\flat})=(-\pi^{\sharp})^{-1}:\wedge^{\bullet} \F \diffto \wedge^{\bullet} \F^*.\] 
Therefore there is a long exact sequence 
 \begin{equation}\label{les poisson}\arraycolsep=1.4pt
  \begin{array}{ccccccccccc}
   \dots &\xrightarrow{\partial}& H^k(\mathcal{F}) &\xrightarrow{j}&H^k(M,\pi)&\xrightarrow{p}&H^{k-1}(\mathcal{F},\nu)&\xrightarrow{\partial}&H^{k+1}(\mathcal{F})&\to&\dots
  \end{array}
 \end{equation}
If $\F$ is coorientable and unimodular, and let $\varphi$ be a closed defining one-form. Then \eqref{ses poisson} can be rewritten as the short exact sequence:
 \begin{align}\label{ses poisson unimod}
  0\to (\varphi\wedge \Omega^{\bullet}(M),\diff )\xrightarrow{j_{\varphi}} (\mathfrak{X}^{\bullet}(M),\diff _{\pi})\xrightarrow{p_{\varphi}} (\varphi\wedge \Omega^{\bullet-1}(M),\diff )\to 0.
 \end{align}
The maps $j_{\varphi}$ and $p_{\varphi}$ can be made explicit as follows: let $V$ be a vector field on $M$ such that $i_V\varphi=1$, and let $\widetilde{\omega}\in \Omega^2(M)$ be the unique extension of $\omega$ such that $i_{V}\widetilde{\omega}=0$. 
Then 
\[j_{\varphi}=(-\pi^{\sharp})\circ i_{V}, \ \ \ \  p_{\varphi}=e_{\varphi}\circ (-\widetilde{\omega}^{\flat})\circ i_{\varphi},\]
where $e_{\varphi}(-)=\varphi\wedge (-)$ is the exterior product with $\varphi$. Even though it was convenient to use $V$ (and $\widetilde{\omega}$) to write these formulas, the maps $j_{\varphi}$ and $p_{\varphi}$ are independent of this choice. However, $V$ allows us to build dual maps:
\[  0\leftarrow \varphi\wedge \Omega^{\bullet}(M)\xleftarrow{p_{V}} \mathfrak{X}^{\bullet}(M)\xleftarrow{j_{V}} \varphi\wedge \Omega^{\bullet-1}(M)\leftarrow 0,\]
with
\begin{equation}\label{dual_maps}
p_V=e_{\varphi}\circ (-\widetilde{\omega}^{\flat}), \ \ \ j_V=e_V\circ (-\pi^{\sharp})\circ i_V,
\end{equation}
which satisfy the homotopy relations:
\begin{equation}\label{dual_maps_relations}
p_V\circ j_V=0,\ \ p_V\circ j_{\varphi}=\mathrm{Id}, \ \ p_{\varphi}\circ j_{V}=\mathrm{Id}, \ \ \mathrm{Id}=j_V\circ p_{\varphi}+j_{\varphi}\circ p_V.
\end{equation}
It can be checked that the maps $p_V$ and $j_V$ are chain morphisms precisely when $V$ is a Poisson vector field, which is also equivalent $\widetilde{\omega}$ being closed; in this case the pair $(\varphi,\widetilde{\omega})$ is a cosymplectic structure on $M$. In general, we can write $\diff \widetilde{\omega}=\varphi\wedge \xi$, where $\xi=i_V\diff \widetilde{\omega}$.

\section{The Heisenberg Lie algebra}\label{sec:heisenberg}
In this section we compute the Poisson cohomology groups of the linear Poisson structure
\begin{align*}
    \pi=z\partial_x\wedge \partial_y
\end{align*}
associated to the Heisenberg Lie algebra.
\begin{proof}[Proof of \cref{thm:heisenberg}]
We do the computations degreewise.

\underline{In degree $0$ :} For $g \in C^\infty(\R^3)$ we obtain for the Poisson differential
\begin{equation}\label{eq: h0}
    \mathrm{d}_\pi g =  z(\partial_y g) \partial_x-z(\partial_x g)\partial_y.
\end{equation}
Therefore $g$ is a Poisson cocycle iff $\partial_x g(x,y,z)=\partial_y g(x,y,z)=0$, since it must hold on the open dense subset $\{z\ne 0\}\subset \R^3$, which is equivalent to
$g(x, y, z) = g(z)$.

\underline{In degree $1$ :} Let $X \in \mathfrak{X}^1(\R^3)$. By Hadamard's Lemma \ref{lemma: hadamard} we can write
\begin{align*}
    X^x= X^x_0(x,y) +z(X_{1,0}^x(y,z)+xX^x_{1,1}(x,y,z))\qquad \text{ and } \qquad X^y= X^y_0+zX^y_1
\end{align*}
and by a choice of $g(x,y,z)= g_0(y,z)+xg_1(x,y,z)$ using \cref{lemma: hadamard} and \eqref{eq: h0} we may assume that
\begin{align*}
    X^x_{1,0}=X^y_1=0.
\end{align*}
Using \eqref{eq: general 1x}-\eqref{eq: general 1z} and taking into account the exchanged roles of $y$ and $z$, we obtain for the Poisson differential of $X$ that
\begin{equation}\label{eq: h1}
    W^x(\pi)=-z\partial_xX^z, \qquad W^y(\pi)=-z\partial_yX^z, \qquad \text{ and }\qquad W^z(\pi)= z(\partial_x X^x + \partial_y X^y_0) - X^z.
\end{equation}
Setting $W^k(\pi)=0$, the first two equations imply that \begin{align*}
    X^z(x,y,z)=X^z(z).
\end{align*}
Using Lemma \ref{lemma: hadamard} we can write
\begin{align*}
    X^z(z)=X^z_0+zX^z_{1,0}+z^2X_{1,1}^z(z)
\end{align*}
and the third equation becomes
\begin{align*}
    0= - X^z_0+ z(\partial_x X^x_0(x,y) + \partial_y X^y_0(x,y)- X^z_{1,0})+z^2(\partial_x (xX^x_{1,1}(x,y,z))-X^z_{1,1}(z))
\end{align*}
which implies that
\begin{align*}
    X^z_0=0,  \qquad  \partial_x X^x_0(x,y) =- \partial_y X^y_0(x,y)+X_{1,0}^z\qquad \text{ and }\qquad X^z_{1,1}(z)=\partial_x (xX^x_{1,1}(x,y,z)).
\end{align*}
The second equations implies that $X^x_0(x,y)$ and $X^y_0(x,y)$ are uniquely determined by a function $g\in C^{\infty}(\R^2)$ and $X^z_{1,0}\in \R$ via
\begin{align*}
    g(0)=0, \qquad  X^x_0(x,y)=\partial_yg(x,y)+xX^z_{1,0}, \qquad \text{ and }\qquad X^y_0(x,y)=-\partial_xg(x,y).
\end{align*}
The third equation implies that
\begin{align*}
    X^z_{1,1}(z)=X_{1,1}^x(x,y,z).
\end{align*}
{Finally, we can change the representative using \eqref{eq: h0} for the function $-\frac{xy}{2}X^z_{1,1}(z)$ and replacing $g(x,y)$ by $g(x,y)+\frac{xy}{2}X^z_{1,0}$ to obtain the desired result.} 

\underline{In degree $2$ :} For $W \in \mathfrak{X}^2(\R^3)$ we obtain by \eqref{eq: general 2} for the Poisson differential that
\begin{align}\label{eq: h2}
    \mathrm{d}_\pi W = z(\partial_yW^x-\partial_x W^y )\partial_x \wedge \partial_y \wedge \partial_z.
\end{align}
Using Lemma \ref{lemma: hadamard} we can write
\begin{align*}
    X^z(x,y,z)=&\, X^z_{0,0}(x,y)+ z(X^z_{1,0}(y,z)+xX^z_{1,1}(x,y,z)),\\
    W^x(x,y,z)=&\, W^x_0(x,y)+z(W^x_{1,0}(x,y)+zW_{1,1}^x(x,y,z)) \\
    W^y(x,y,z)=&\, W^y_0(x,y)+z(W^y_{1,0}(x,y)+z(W_{1,1,0}^y(y,z)+xW^y_{1,1,1}(x,y,z)))
\end{align*}
Hence by equation \eqref{eq: h1} we may assume that
\begin{align*}
    W^x_{1,1}(x,y,z)=W^y_{1,1,0}(y,z)=W^z(x,y,z)=0. 
\end{align*}
Here we used invertability of $\partial_x$ according to Corollary \ref{cor: invertible}. The cocyle condition for $W$ then becomes
\begin{align*}
    \partial_y (W^x_0(x,y)+zW^x_{1,0}(x,y))=\partial_x (W^y_0(x,y)+zW^y_{1,0}(x,y)+xz^2W^y_{1,1,1}(x,y,z)).
\end{align*}
which implies by comparing powers in $z$ that
\begin{align*}
    W^x=\partial_x(g_0(x,y)+zg_1(x,y))\qquad \text{ and }\qquad W^y=\partial_y(g_0(x,y)+zg_1(x,y)) 
\end{align*}
using again invertability of $\partial_x$ as mentioned above.

\underline{In degree $3$:} Using that the Poisson differential of every three-vector field in $\R^3$ is zero, equation \eqref{eq: h2} together with Lemma \ref{lemma: hadamard} implies the result.

\underline{The algebraic structure:} The relation for the algebraic structure follow from \eqref{eq: h0}, \eqref{eq: h1}, \eqref{eq: h2} and \cref{lemma: hadamard}.
\end{proof}

\section{The Lie algebra $\mathfrak{aff}(1,\mathbb{R})\, \mathrm{x}\, \mathbb{R}$}\label{sec:affine}
In this section we present a calculation for the lie algebra $\mathfrak{aff}(1,\R) \times \R$, i.e.
\begin{align*}
    \pi=x\partial_x\wedge \partial_y
\end{align*}
\begin{proof}[Proof of \cref{thm:affine3}]
\underline{In degree $0$:} For $g \in C^{\infty}(\R^3)$ we obtain
\begin{align}\label{eq: a0}
    \mathrm{d}_\pi g = x(\partial_y g)\partial_x-x(\partial_x g)\partial_y.
\end{align}
Hence $\mathrm{d}_\pi g = 0$ iff $\partial_xg = \partial_yg = 0$ which implies that $g(x,y,z)=g(z)$.

\underline{In degree $1$:} Let $X \in \mathfrak{X}(\R^3)$. By \eqref{eq: general 1x}-\eqref{eq: general 1z} and taking into account that we exchanged the roles of $y$ and $z$ we obtain for the Poisson differential of $X$ that
\begin{align}\label{eq: a1}
    W^x(\pi)=-x\partial_xX^z, \qquad W^y(\pi)=-x\partial_yX^z, \qquad \text{ and }\qquad 
    W^z(\pi)= x(\partial_xX^x + \partial_yX^y) - X^x
\end{align}
Using Hadamard's Lemma \ref{lemma: hadamard} and \eqref{eq: a0} we may assume that
\begin{align*}
    X^x(x,y,z)=X^x_0(y,z)\qquad \text{ and } \qquad X^y(x,y,z)=X^y_0(y,z)+xyX^y_{1,1}(x,y,z).
\end{align*}
The first two equations in \eqref{eq: a1} imply that $X^z(x,y,z)=X^z(z)$. The third equation becomes
\begin{align*}
    0= -X_0^x(y,z)+x\partial_y(X^y_0(y,z)+xyX^y_{1,1}(x,y,z)) 
\end{align*}
which implies, by separating the powers of $x$ and using Corollary \ref{cor: invertible}, that
\begin{align*}
    X_0^x(y,z)=\partial_yX^y_0(y,z)=X^y_{1,1}(x,y,z)=0.
\end{align*}
Hence the result follows.

\underline{In degree $2$:} For $W \in \mathfrak{X}^2(\R^3)$ we obtain by \eqref{eq: general 2} and adapted accordingly, that
\begin{align}\label{eq: a2}
    \mathrm{d}_\pi W = (x(\partial_y W^x-\partial_x W^y) + W^y) \partial_x\wedge\partial_y\wedge\partial_z.
\end{align}
By Hadamard's Lemma \ref{lemma: hadamard} and \eqref{eq: a1} we may assume that
\begin{align*}
    W^x(x,y,z)=W^x_0(y,z)+xyW^x_{1,1}(x,y,z), \qquad  W^y(x,y,z)=W^y_0(y,z)\qquad \text{ and } \qquad W^z(x,y,z)=0.
\end{align*}
Then the cocycle condition for $W$ becomes
\begin{align*}
    0=W^y_0(y,z)+x\partial_y(W^x_0(y,z)+xyW^x_{1,1}(x,y,z))
\end{align*}
which implies similar as above, by separating in powers of $x$ and Corollary \ref{cor: invertible}, that
\begin{align*}
    W^y_0(y,z)=\partial_yW^x_0(y,z)=W^x_{1,1}(x,y,z)=0
\end{align*}
and hence we obtain the result.

\underline{In degree $3$:} Equation \eqref{eq: a2} together with Lemma \ref{lemma: hadamard} implies that the Poisson differential is surjective onto three-vector fields.
\end{proof}

\section{The Euclidean Lie algebra $\mathfrak{e}(2)$}\label{sec:euclidean}
In this section we compute the Poisson cohomology is the Euclidean algebra $\mathfrak{e}(2)$. The Poisson structure on $\mathfrak{e}(2)^\ast$ is given by 
\begin{align*}
    \pi = T\wedge \partial_z\qquad \text{ where } \ T:= - y\partial_x+x\partial_y.
\end{align*}
The idea is to study the subcomplex $(\mathfrak{X}^{\bullet,S^1}(\mathfrak{e}(2)^\ast),\mathrm{d}_\pi)$ of $(\mathfrak{X}^\bullet(\R^3), \mathrm{d}_\pi)$
of multivector fields invariant under the standard $S^1$-action in the $(x,y)$-plane, i.e. the action with infinitesimal vector field $T$. Note that $\pi = T \wedge \partial_z$ is invariant under the action of $S^1$, hence the differential $\mathrm{d}_\pi$ descends to $\mathfrak{X}^{\bullet,S^1}(\mathfrak{e}(2)^\ast)$,
giving rise to the \emph{invariant cohomology} $H^{\bullet,S^1}(\R^3, \pi)$. Then we use averaging with respect to the $S^1$ on $(\mathfrak{e}(2)^\ast,\pi)$. This is based on a result by Ginzburg (see Proposition 4.11 in~\cite{Ginzburg1996}):
\begin{proposition}\label{proposition: ginzburg}
If a compact Lie group $K$ acts on a Poisson manifold $(M,\pi)$ by Hamiltonian diffeomorphisms, then we obtain an isomorphism
\begin{align*}
    H^{\bullet,K}(M,\pi) \cong H^\bullet(M, \pi).
\end{align*}
\end{proposition}
\begin{remark}
    The statement of Ginzburg holds in the more general setting of actions by closed cotangent lifts, i.e. if the infinitesimal action $a: \mathfrak{k}\to \mathfrak{X}^1(M)$ satisfies $a=\pi^{\sharp}\circ \bar{a}$ where $\bar{a}:\mathfrak{k} \to \Omega^1_{\text{closed}}(M)$.
\end{remark}

In the next subsections we first determine the $S^1$-invariant multi-vector fields and then compute the invariant cohomology.

\subsection{Classification of $S^1$-invariant multivector fields on $\R^3$}
We consider the standard $S^1$-action $A:S^1\times \R^3\to \R^3 $ given by
\begin{equation*}
	A_\theta ((x,y,z))= (x\cos \theta - y\sin\theta, x\sin\theta + y\cos\theta, z),
\end{equation*}
infinitesimally generated by $a(\partial_\theta)= T$.
\begin{definition}
	We say that $V\in \mathfrak{X}^{\bullet}(\R^3)$ is $S^1$-invariant if for all $\theta \in S^1$ we have
	\begin{equation}\label{eq: def-invariant}
	    A_\theta^*(V) = V 
	\end{equation}
	The space of all  $S^1$\emph{-invariant multivector fields}\index{Multivector field!invariant} is denoted by $\mathfrak{X}^{\bullet,S^1}(\R^3)$.
\end{definition}
\begin{remark}
    Since $S^1$ is connected \eqref{eq: def-invariant} is equivalent to
	\begin{align*}
	    \lie _T V=0
	\end{align*}
\end{remark}
To classify all $S^1$-invariant multivector fields on $\R^3$ the following lemmas are useful.
\begin{lemma}\label{lemma:even_function_ring_isomorphism}
The map $\varphi: C^\infty([0,\infty)) \rightarrow C^\infty_{\mathrm{even}}(\R)$ defined by
\begin{equation*}
    \varphi(f) = (x \mapsto f(x^2)).
\end{equation*}
is an isomorphism of rings.
\end{lemma}
\begin{proof}
Linearity and preservation of the product are clear. Furthermore $\ker \varphi = \Set{0}$, thus it is injective.

For surjectivity, let $g \in C^\infty_{\textrm{even}}(\R)$ and define $f(x) \coloneqq g(\sqrt{x})$. Then $\varphi(f)(x) = g(\sqrt{x^2}) = g(|x|) = g(x)$ using that $g$ is even.
Therefore, we only have to proof that $f \in C^\infty([0,\infty))$. Note that $f$ is continuous as the composition of continuous functions.
Thus it remains to show that all derivatives exist and are continuous. For this we introduce $t(x) \coloneqq \sqrt{x}$. Then
\begin{equation*}
    \frac{d}{dx} = \frac{dt}{dx}\frac{d}{dt} = \frac{1}{2\sqrt{x}}\frac{d}{dt} = \frac{1}{2t}\frac{d}{dt}.
\end{equation*}
    Since $g$ is even, we have $g'(0) = 0$, hence by Hadamard's lemma \ref{lemma: hadamard} there exists a smooth function $h$ such that
\begin{equation*}
    g'(y) = g'(0) + yh(y) = yh(y).
\end{equation*}
Since $g$ is even, $g'(y)$ is odd, and therefore $g'(-y) = (-y)h(-y) = -yh(y)$, thus $h(-y) = h(y)$ and therefore $h \in C^\infty_{\textrm{even}}(\R)$.
In particular $h$ is continuous and therefore also $h \circ t = h(t)$ is continuous. Then the first derivative of $f$ is given by
\begin{equation*}
    \frac{df}{dx} = \frac{d}{dx}g(t) = \frac{1}{2t}\frac{d}{dt}g(t) = \frac{g'(t)}{2t} = \frac{th(t)}{2t} = \frac{1}{2}h(t),
\end{equation*}
    hence $\frac{df}{dx}$ exists and is continuous, so $f \in C^1([0,\infty))$. Now, since $\frac{1}{2}h$ is an element of $C^\infty_{\textrm{even}}(\R)$, we can
apply the same argument to $f'(x) = f^{(1)}(x) \coloneqq \frac{1}{2}h(\sqrt{x})$, hence inductively we find that all derivatives exist and are smooth. Thus $f \in C^\infty([0,\infty))$.
\end{proof}

\begin{lemma}\label{lemma:invariant_under_T}
Let $g \in C^\infty(\R^3)$. Then $g$ is $S^1$-invariant iff there exists a smooth function $f \in C^\infty([0,\infty) \times \R)$ such that $g(x,y,z) = f(x^2 + y^2, z)$
\end{lemma}
\begin{proof}
Define $h: \R^3 \rightarrow \R$ by
\begin{equation*}
    h: (s,t,z) \mapsto f(t\cos s, t\sin s,z).
\end{equation*}
Then we obtain that
\begin{align*}
    h(\pi, t,z) =&\,  g(-t,0,z) = h(0, -t,z) \qquad \text{ and } \qquad  \partial_s h = (\lie _T f)(t\cos s, t\sin s, z) = 0 
\end{align*}
for all $(s,t,z)\in \R^3$. The second equations implies that $h(s,t,z) = \tilde{h}(t,z)$ for $\tilde{h}:\R^2\to \R$ and the first equation implies that $\tilde{h}(\cdot,z)$ is even for every $z$. Finally, for $(x,y,z) \in \R^3$, define $r \coloneqq \sqrt{x^2 + y^2}$, then there exists $\theta$ such that 
$(x,y) = (r\cos \theta, r\sin \theta)$ and hence
\begin{equation*}
    g(x,y,z) = \tilde{h}(r,z)
\end{equation*}
By the previous lemma there exists one parameter family $f:\R\to C^\infty([0,\infty))$ such that 
\begin{equation*}
    \tilde{h}(\sqrt{x^2 + y^2},z) = f_z(x^2 + y^2),
\end{equation*}
Finally, since $g$ is smooth we obtain that $f\in C^\infty([0,\infty) \times \R)$ with $g(x,y,z) = f(x^2 + y^2, z)$.
\end{proof}

\begin{lemma}\label{lemma:second_order_invariance}
	Let $g \in C^\infty(\R^3)$ and suppose $T^2g = -g$, then there exist smooth functions
	$g_+, g_- \in C^\infty([0,\infty) \times \R)$ such that
	\begin{equation*}
		g(x,y,z) = xg_+(x^2 + y^2, z) + yg_-(x^2 + y^2, z).
	\end{equation*}
\end{lemma}
\begin{proof}
	Let $\psi: \R \times S^1 \times \R \rightarrow \R^3$ be defined by $ (r,\theta,z) \mapsto (r\cos \theta, r\sin \theta, z)$. Then $\partial_\theta$ is $\psi$-related to $T$. Set $\hat{g} = \psi^\ast g$. The condition
	on $g$ gives $\partial^2_\theta \hat{g} = -\hat{g}$. For fixed $r,z \in \R$ consider $\hat{g}_{r,z}(\theta) := \hat{g}(r, \theta, z)$. The equation $\hat{g}_{r,z}'' = -\hat{g}_{r,z}$ implies that 
	\begin{equation*}
		\hat{g}(r,\theta,z) = \hat{g}_+(r,z)\cos\theta + \hat{g}_-(r,z)\sin\theta.
	\end{equation*}
	for some $\hat{g}_\pm :\R^2\to \R $. Since $\hat{g}_+ = \hat{g}(\cdot,0,\cdot)$ and $\hat{g}_- = \hat{g}\left(\cdot,\frac{\pi}{2},\cdot\right)$ we obtain that $\hat{g}_+, \hat{g}_- \in C^{\infty}(\R^2)$.

	Note that we have $\psi(-r, \theta, z) = \psi(r, \theta + \pi, z)$. Therefore, we must have
	\begin{equation*}
	\hat{g}_\pm (-r,z) = \hat{g}\left(-r,\frac{\pi\pm (-\pi)}{4},z\right) = g\left(\psi\left(-r,\frac{\pi\pm (-\pi)}{4},z\right)\right) = \hat{g}\left(r,\frac{\pi\pm (-\pi)}{4} + \pi, z\right) = -\hat{g}_\pm(r,z).
	\end{equation*}
    So both are odd in $r$, and thus $\hat{g}^\pm(0,z) = 0$. By Hadamard's lemma \ref{lemma: hadamard} we may
	write
	\begin{equation*}
		\hat{g}_\pm(r,z) = \hat{g}_{\pm,0}(0,z) + r\hat{g}_{\pm,1}(r,z) = r\hat{g}_{\pm,1}(r,z),
	\end{equation*}
	with $\hat{g}_{\pm,1} \in C^\infty(\R^2)$ even in $r$. Therefore, by \cref{lemma:even_function_ring_isomorphism} there exist functions $\hat{g}_\pm \in C^\infty([0,\infty) \times \R)$ such that $\hat{g}_{\pm,1}(r,z) = \hat{g}_\pm(r^2, z)$ and hence we have
	\begin{equation*}
		g(x,y,z) = xg_+(x^2 + y^2, z) + yg_-(x^2 + y^2, z).\qedhere
	\end{equation*}
\end{proof}

We can now prove the main result of this section.

\begin{proposition}[Classification]\label{prop:classification_invariants}
Let $S^1$ act on $\R^3$ by rotation around the $z$-axis and recall that $E = x\partial_x + y\partial_y$. The $S^1$-invariant multivector fields are given by
\begin{itemize}
    \item the $S^1$-invariant smooth functions on $\R^3$ are of the form $g(x^2 + y^2, z)$ for $g\in C^\infty([0,\infty) \times \R)$
    \item the $S^1$-invariant vector fields form a free module over $C^{\infty,S^1}(\R^3)$ with three generators:
    \begin{align*}
        \langle E,T,\partial_z \rangle
    \end{align*}
    \item similarly, the $S^1$-invariant bivector fields form a free module three generators:
    \begin{align*}
        \langle E\wedge \partial_z,T\wedge \partial_z,\partial_x\wedge\partial_y \rangle
    \end{align*}
    \item the $S^1$-invariant threevector fields form a free module with generator:
    \begin{align*}
        \langle \partial_x\wedge \partial_y\wedge \partial_z \rangle
    \end{align*}
\end{itemize}
\end{proposition}
\begin{proof}
	Note that all maps given are well-defined, injective module homomorphisms with respect to $C^\infty([0,\infty))$ and $C^{\infty,S^1}(\R^3)$. Therefore we only need to check surjectivity.

	\underline{In degree $0$:} If $g \in \mathfrak{X}^{0,S^1}(\R^3)$ then $\mathscr{L}_T g = 0$
	and hence $g(x,y,z)= f(x^2 + y^2, z)$ by \cref{lemma:invariant_under_T}.

	\underline{In degree $1$:} For $X \in \mathfrak{X}^{1,S^1}(\R^3)$ the condition $\lie_{T} X = 0$ is equivalent to the equations
    \begin{align}\label{eq: invariant}
        \lie_T X^x + X^y = \lie_T X^y - X^x =\lie_T X^z = 0.
    \end{align}
	For $X^z$ the result is implied by \cref{lemma:invariant_under_T}. The equations for $X^x$ and $X^y$ imply that
	\begin{align*}
	    T^2 X^k=-X^k 
	\end{align*}
	for $k=x,y$ and hence by \cref{lemma:second_order_invariance} we obtain
	\begin{align*}
	    X^k(x,y,z)= xX^k_+(x^2+y^2,z)+yX^k_-(x^2+y^2,z).
	\end{align*}
	Therefore \eqref{eq: invariant} implies the result.

	\underline{In degree $2$:}
    Let $W \in \mathfrak{X}^{2,S^1}(\R^3)$. The $S^1$-invariance for $W$ is equivalent to the equations
    \begin{align*}
        \lie_T W^x + W^y=-\lie_ T W^y + W^x =\lie_ T W^z = 0.
    \end{align*}
    Hence the result follows as in degree $1$.

	\underline{In degree $3$:}  
    A three-vector field $h\partial_x \wedge \partial_y \wedge \partial_z \in \mathfrak{X}^{3,S^1}(\R^3)$ is invariant iff
    \begin{align*}
        \lie_{T} h = 0.
    \end{align*}
    and hence \cref{lemma:invariant_under_T} implies the result.
\end{proof}

\subsection{Invariant Poisson cohomology}
In this section we compute $H^{\bullet, S^1}(\mathfrak{e}(2)^\ast,\pi)$ and prove \cref{thm:cohomology_euclidean}. 

\begin{theorem}[Invariant cohomology of $\mathfrak{e}(2)^\ast$]
	Let $\mathfrak{e}(2)$ be the Euclidean Lie algebra of dimension 3 with the corresponding Poisson structure $(\mathfrak{e}(2)^\ast, \pi = T\wedge \partial_z)$. The invariant cohomology $H^{\bullet, S^1}(\mathfrak{e}(2)^\ast,\pi)$ is described as follows:
	\begin{itemize}
		\setlength\itemsep{0em}
		\item in degree $0$, each cohomology class is uniquely determined by
			\begin{equation*}
				g(x^2 + y^2)
			\end{equation*}
			for $g \in C^\infty([0,\infty))$.
		\item in degree $1$, the cohomology classes are uniquely represented by vector fields of the form
			\begin{equation*}
				X^E(x^2 + y^2)E + X^z(x^2 + y^2)\partial_z
			\end{equation*}
			for $X^E,X^z \in C^\infty([0,\infty))$.
		\item in degree $2$, the cohomology classes can be uniquely represented by
			\begin{equation*}
				W^T(x^2 + y^2) E \wedge \partial_z + W^z(z)\partial_x \wedge \partial_y.
			\end{equation*}
			for $W^T \in C^\infty([0,\infty))$ and $W^z \in C^\infty(\R)$.
		\item in degree $3$, the cohomology classes are uniquely represented by
			\begin{equation*}
				h(z)\partial_x\wedge\partial_y\wedge\partial_z
			\end{equation*}
			with $h \in C^\infty(\R)$.
	\end{itemize}
\end{theorem}
\begin{proof}
    We use $S^1$-invariant mulitvector fields as described in Proposition \ref{prop:classification_invariants}.
    
	\underline{In degree $0$:} The Poisson differential for $g \in \mathfrak{X}^{0,S^1}(\R^3)$ is given by
	\begin{equation}\label{eq: e0}
		\mathrm{d}_\pi g = \partial_z g T.
	\end{equation}
	Hence $g$ is a cocyle iff $g$ is independent of $z$.

	\underline{In degree $1$:} Let $X \in \mathfrak{X}^{1,S^1}(\R^3)$. For the Poisson differential of $X$ we obtain by a direct computation that
	\begin{equation}\label{eq: e1}
	    \begin{aligned}
	        W^x(\pi)= x\partial_zX^z(x^2 + y^2, z),&\, \qquad 
	        W^y(\pi)= y\partial_zX^z(x^2 + y^2, z)\\
	        W^z(\pi)=&\, -(x^2 + y^2)\partial_z X^E(x^2 + y^2, z)
	    \end{aligned}
	\end{equation}
	By \eqref{eq: e0} we may assume that $X^T=0$ and hence the cocycle condition for $X$ is equivalent to
	\begin{equation*}
	    \partial_z X^E(x^2 + y^2, z)=\partial_zX^z(x^2 + y^2, z) = 0.
	\end{equation*}
	which implies that $X^E$ and $X^z$ are independent of $z$.

    \underline{In degree $2$:} For $W \in \mathfrak{X}^{2,S^1}(\R^3)$ we obtain for the Poisson differential: 
    \begin{align}\label{eq: e2}
        \mathrm{d}_\pi W = -\partial_z W^T(r^2, z)T \wedge E \wedge \partial_z.
    \end{align}
    Note that \eqref{eq: e1} is equivalent to saying
    \begin{align*}
        W(\pi)= \partial_zX^z(x^2+y^2,z)T\wedge\partial_z -(x^2 + y^2)\partial_z X^E(x^2 + y^2, z)\partial_x\wedge \partial_y. 
    \end{align*}
    Since $[0,\infty) \times \R$ is convex, Hadamard's lemma \ref{lemma: hadamard} applies and we may assume that
    \begin{align*}
        W^E(x^2+y^2,z)=0 \qquad \text{ and }\qquad W^z(x^2+y^2,z)=W^z(z)
    \end{align*}
    Moreover, the cocycle condition for $W$ is equivalent to $\partial_z W^T(r^2, z)=0$, which implies the result.

	\underline{In degree $3$:} Equation \eqref{eq: e2} together with Hadamard's Lemma \ref{lemma: hadamard} imply the result in this degree.
\end{proof}

We can now prove the main result of this section.
\begin{proof}[Proof of \cref{thm:cohomology_euclidean}]
    By Proposition \ref{proposition: ginzburg} we have 
	\begin{equation*}
		H^{\bullet,S^1}(\mathfrak{e}(2)^\ast, \pi) \cong H^\bullet(\mathfrak{e}(2)^\ast, \pi).\qedhere
	\end{equation*}
\end{proof}

\section{Open book and hyperbolic-type Lie algebras}\label{sec:typevi}
In this section we give a proof for the Poisson cohomology of the Poisson structure $\pi_{\tau}$ associated with the Lie algebra $\mathfrak{b}_\tau$ for $0 < |\tau| \leq 1$. Let $E_\tau = x\partial_x + \tau y\partial_y$, then the Poisson structures $\pi_{\tau}$ are given by
\begin{align*}
    \pi_\tau = E_\tau \wedge \partial_z
\end{align*}
We make use of the first short exact sequence in \cref{proposition: ses} to obtain the long exact sequence
\begin{equation}\label{eq: les R}
\ldots \stackrel{j^{\infty}_R}{\to} H^{q-1}_{\Rho _F}(\R^3,\pi_{\tau})\stackrel{\partial}{\to} 
H^{q}_{\Rho _f}(\R^3,\pi_{\tau})\to H^{q}(\R^3,\pi_{\tau})\stackrel{j^{\infty}_R}{\to}H^{q}_{\Rho _F}(\R^3,\pi_{\tau} )\stackrel{\partial}{\to}\ldots. \end{equation}

\subsection{Formal Poisson cohomology}
To compute the formal Poisson cohomology for $0 < |\tau| \leq 1$ we distinguish the cases $\tau = -\frac{p}{q}$ for $p,q\in \N$ relatively prime (including $\tau=-1$), $\tau = 1$, $\tau = \frac{1}{n}$ for some integer $n \geq 2$ and $\tau \notin [-1,0)\backslash \Q \cup (0,1] \backslash \left\{\frac{1}{n}\right\}_{n\in \N} $.
\begin{theorem}\label{thm:open_book_formal}
    For $0 < |\tau | \leq 1$ the formal Poisson cohomology groups $H_{\Rho _F}^{\bullet}(\mathfrak{b}_{\tau}^*,\pi_{\tau})$ can be described as follows:
    \begin{itemize}
        \item in degree $0$ we have the canonical isomorphisms 
        \begin{align*}
            H_{\Rho _F}^{0}(\mathfrak{b}_{\tau}^*,\pi_{\tau})\simeq \begin{cases}
            \R[[x^py^q]] &\text{ if } \tau =-\frac{p}{q} \ \text{ where } p\le q\in \N \ \text{ are relatively prime}\\
            \R & \text{ else;}
            \end{cases}
        \end{align*}
        \item in degree $1$ we have an isomormorphism
         \begin{align*}
            H_{\Rho _F}^{1}(\mathfrak{b}_{\tau}^*,\pi_{\tau})\simeq \begin{cases}
            \R[[x^py^q]] &\text{ if } \tau =-\frac{p}{q} \ \text{ where } p\le q\in \N \ \text{ are relatively prime}\\
            \R^4 & \text{ if }\tau=1\\
            \R^3 &\text{ if } \tau=\frac{1}{n} \text{ with } 2\le n\in \N\\
            \R^2 & \text{ else.}
            \end{cases}
        \end{align*}
        Explicit representatives are given in the different cases by
        \begin{align*}
            f_1E +f_2\partial_z \qquad \qquad \qquad &\text{ for } f_1,f_2\in \R[[x^py^q]], \text{ and }\tau =-\frac{p}{q} \ \text{ with } p\le q\in \N \ \text{ relatively prime}\\
            ay\partial_x + (bx+cy)\partial_y + d\partial_z \qquad &\text{ for } a,b,c,d\in \R , \text{ and }\tau =1 \\
            ay^n\partial_x + by\partial_y + c\partial_z \qquad \quad &\text{ for } a,b,c\in \R, \text{ and }\tau =\frac{1}{n} \text{ with } 2\le n\in \N \\
            ay\partial_y + b\partial_z \qquad \qquad &\text{ for } a,b\in \R, \text{ else.}
        \end{align*}
        \item in degree $2$ we have an isomormorphism
         \begin{align*}
            H_{\Rho _F}^{2}(\mathfrak{b}_{\tau}^*,\pi_{\tau})\simeq \begin{cases}
            \R[[x^py^q]] &\text{ if } \tau =-\frac{p}{q} \ \text{ where } p\le q\in \N \ \text{ are relatively prime}\\
            \R^3 & \text{ if }\tau=1\\
            \R^2 &\text{ if } \tau=\frac{1}{n} \text{ with } 2\le n\in \N\\
            \R & \text{ else.}
            \end{cases}
        \end{align*}
        Explicit representatives are given in the different cases by
        \begin{align*}
            fE\wedge \partial_z \qquad \qquad \qquad &\text{ for } f\in \R[[x^py^q]], \text{ and }\tau =-\frac{p}{q} \ \text{ with } p\le q\in \N \ \text{ relatively prime}\\
            ay\partial_x\wedge \partial_z + (bx+cy)\partial_y\wedge \partial_z  \qquad &\text{ for } a,b,c\in \R , \text{ and }\tau =1 \\
            ay^n\partial_x\wedge \partial_z + by\partial_y\wedge\partial_z  \qquad \quad &\text{ for } a,b\in \R, \text{ and }\tau =\frac{1}{n} \text{ with } 2\le n\in \N \\
            ay\partial_y\wedge \partial_z \qquad \qquad \qquad &\text{ for } a\in \R, \text{ else.}
        \end{align*}
        \item in degree $3$ we have an isomormorphism
        \begin{align*}
            H_{\Rho _F}^{3}(\mathfrak{b}_{\tau}^*,\pi_{\tau})=0 \qquad \text{ for all } \ 0<|\tau|\le 1
        \end{align*}
    \end{itemize}
    Moreover the algebraic structures are given as follows:
    \begin{itemize}
        \item if $\tau =-\frac{p}{q}$ with $p\le q\in \N$ relatively prime, then the algebraic operations map representatives onto representatives;
        \item if $\tau=1$ the algebraic structure is determined by the relations
        \begin{align*}
            \left[x\partial_x\wedge\partial_z-y\partial_y\wedge\partial_z\right]=\left[0\right]\in H^2_{\Rho _F}(\mathfrak{b}_{1}^*,\pi_{1})\qquad \text{ and }\qquad \left[p(x,y)\partial_x\wedge\partial_y\right] =\left[0\right]\in H^2_{\Rho _F}(\mathfrak{b}_{1}^*,\pi_{1})
        \end{align*}
        for any polynomial $p$ homogeneous of degree 2;
        \item if $\tau=\frac{1}{n}$ for $2\le n\in \N$ then we have the relation
        \begin{align*}
            \left[y^{n+1}\partial_x\wedge\partial_y\right]=\left[0\right]\in H^2_{\Rho _F}(\mathfrak{b}_{\frac{1}{n}}^*,\pi_{\frac{1}{n}});
        \end{align*}
        otherwise the image of the operations on representatives is again a representative, except when it maps into $H^3_{\Rho _F}(\mathfrak{b}_{\tau}^*,\pi_{\tau})$ where it is zero.
        \item otherwise the algebraic operations map representatives to representatives.
    \end{itemize}
\end{theorem}

For this the following lemma is useful.
\begin{lemma}\label{lemma: number game}
	Let $0 < |\tau | \leq 1$ and $i,j \in \N_0$. Then
	\begin{enumerate}[(i)]
		\setlength\itemsep{-0.2em}
		\item If $\tau = \frac{1}{n}$ for some $n\in \N$, then $i + \tau j - 1 = 0$ iff $i = 0, j = n$ or $i = 1, j = 0$.
		\item If $0<\tau \neq \frac{1}{n}$ for all positive integers, then $i + \tau j - 1 = 0$ iff $i = 1$ and $j = 0$.
		\item If $\tau = -\frac{p}{q}$ for $p\le q\in \N$ relatively prime, then $i + \tau j - 1 = 0$ iff $j=nq, i=np+1$ for $n\in \N_0$. 
	\end{enumerate}
\end{lemma}
\begin{proof}
	(i, ii) The implication from right to left is a simple calculation. For the other implication, suppose $i + \tau j - 1 = 0$, then $\tau j = 1 - i$. We have $\tau j \geq 0$, thus $i \leq 1$. Then with $i \geq 0$ follows that either $i = 0$ or $i = 1$. If $ i = 1$, then $\tau j = 0$ and therefore $j = 0$. If $ i = 0$, then $\tau j = 1$, thus $\tau = \frac{1}{j}$. If $\tau = \frac{1}{n}$, it follows that $j = n$. If $\tau \neq \frac{1}{n}$ for all positive integers $n$, this gives a contradiction, hence there is no solution with $i = 0$ in this case. Finally, (iii) follows from a direct calculation.
\end{proof}

\begin{proof}[Proof of \cref{thm:open_book_formal}]
	\underline{In degree $0$:} \ Let $g \in C^\infty(\R)[[x,y]]$ given by $g(x,y,z) = \sum_{i,j = 0}^\infty g_{ij}(z)x^i y^j$. Using \eqref{eq: general 0} we obtain for the Poisson differential:
	\begin{equation}\label{eq:formal:diff_function}
	    \diff_{\pi_\tau} g = \begin{aligned}[t]& \left(\sum_{i,j = 0}^\infty \partial_zg_{ij}(z) x^{i + 1}y^j\right)\partial_x  +\left(\sum_{i,j = 0}^\infty \tau\partial_zg_{ij}(z) x^iy^{j+1}\right)\partial_y -\left(\sum_{i,j = 0}^\infty \left(i + \tau j\right)g_{ij}(z)x^iy^j\right)\partial_z.\end{aligned}
    \end{equation}
	For $g$ to be a cocyle is equivalent to
    \begin{equation*}
	    0=\mathrm{d}_{\pi_\tau} g  \qquad \Leftrightarrow \qquad \forall i,j : \ \begin{cases}0=\left(i + \tau j\right)g_{ij}(z)\\ 
	    0=\partial_zg_{ij}(z)\end{cases}
    \end{equation*}
	The second equation implies $g_{ij}(z) \equiv g_{ij} \in \R$ for all $i,j \geq 0$ and hence the first equation together with \cref{lemma: number game} imply the result.

    \underline{In degree $1$:} \ Let $X \in \mathfrak{X}_{\Rho _F}^1(\R^3)$. It's Poisson differential $\diff_{\pi_{\tau}}$ is by \eqref{eq: general 1x}, \eqref{eq: general 1y} and \eqref{eq: general 1z} described by
    \begin{align}
	    W^x(\pi_{\tau})=&\, \lie_{E_\tau} X^y - \tau X^y + \tau y\partial_z X^z\label{eq: general 611}\\
	    W^y(\pi_{\tau})=&\, -\lie _{E_\tau} X^x + X^x - x\partial_z X^z \label{eq: general 612} \\
	    W^z(\pi_{\tau})=&\, x\partial_z X^y - \tau y\partial_z X^x \label{eq: general 613} 
    \end{align}
    Here $X^k(x,y,z) = \sum_{i,j = 0}^\infty X^k_{ij}(z)x^i y^j$ for all $k\in \{x,y,z\}$. Following our general strategy for computing formal Poisson cohomology we may, by a choice of $g$ and \eqref{eq:formal:diff_function}, assume that
    \begin{align*}
        X^z_{ij}(z) = &\, 0 \quad \text{ for all } \begin{cases}
        (i,j) \ne (np,nq)  \ \text{ and } \ n \in \N_0 , &\text{ if } \tau =-\frac{p}{q} \ \text{ where } p\le q\in \N \ \text{relatively prime}\\
        (i,j)\ne 0 & \text{ else.}
        \end{cases}\\
        X^x_{ij}(z) = &\, 0 \quad \text{ for all } \begin{cases}
        (i,j) = (np+1,nq)  \ \text{ and } \ n \in \N_0 , &\text{ if } \tau =-\frac{p}{q} \ \text{ where } p\le q\in \N \ \text{relatively prime}\\
        (i,j)= (1,0) & \text{ else.}
        \end{cases}
    \end{align*}
    In the following we will treat the two cases separately. 
    \begin{itemize}
        \item For $\tau=-\frac{p}{q}$ with $p\le q\in \N$ relatively prime:
        
        By setting $W(\pi)=0$ equation \eqref{eq: general 612} becomes equivalent to
        \begin{align*}
           0=(i+\tau j-1)X^x_{ij}(z)+\delta_{i(np+1)}\delta_{j(nq)}\partial_zX^z_{i-1j}(z) \qquad \text{ for all }\ i,j\ge 0
        \end{align*}
        Note that the $X^x_{ij}(z)$ which are zero don't appear anyways, since in these cases the coefficient is zero. Hence this allows us to conclude that:
        \begin{align*}
            X^x_{ij}(z)=0 \qquad \text{ for all }\ i,j\ge 0 \qquad \text{ and } \qquad X^z_{ij}(z)=\begin{cases}
            X^z_{ij}\in \R &\text{ if }\, (i,j) = (np,nq)  \text{ and } \ n \in \N_0 , \\
        0&\text{ else.}\end{cases}
        \end{align*}
        Therefore, \eqref{eq: general 611} and \eqref{eq: general 613} become equivalent to
        \begin{align*}
            0=(i+\tau(j-1))X^y_{ij}(z) \qquad \text{ and } \qquad 0=\partial_zX^y_{ij}(z)
        \end{align*}
        for all $i,j\ge 0$, which implies using \cref{lemma: number game} that
        \begin{align*}
            X^y_{ij}(z)=\begin{cases}
            X^y_{ij}\in \R & \text{ if }(i,j) = (np,nq+1)  \ \text{ where } \ n \in \N_0\\
            0& \text{ else.}
            \end{cases}
        \end{align*}
        \item For the other cases, i.e. $0<|\tau|\le 1$ and $\tau \notin [-1,0)\cap \Q$ we follow the same line of calculation:
        
        Setting $W(\pi)=0$, implies that equation \eqref{eq: general 612} is equivalent to
        \begin{align*}
           0=(i+\tau j-1))X^x_{ij}(z)+\delta_{i1}\delta_{j0}\partial_zX^z_{00}(z) \qquad \text{ for all }\ i,j\ge 0
        \end{align*}
        Hence we conclude using \cref{lemma: number game} that:
        \begin{align*}
            X^x_{ij}(z)=&\, 0 \begin{cases}
            \text{ for all } (i,j)\ne (0,n)&\text{ if } \tau = \frac{1}{n} \text{ for some } n\in \N\\
            \text{ for all }\ i,j\ge 0&\text{ else,}\end{cases} \\
            X^z_{ij}(z)=&\, \begin{cases}
            X^z_{00}\in \R &\text{ if }\, (i,j) = 0 , \\
        0&\text{ else.}\end{cases}
        \end{align*}
        Therefore, \eqref{eq: general 611} and \eqref{eq: general 613} become equivalent to
        \begin{align*}
            0=(i+\tau(j-1))X^y_{ij}(z) \qquad \text{ and } \qquad \begin{cases}
           \delta_{i0}\delta_{(j-1)n}\partial_zX^x_{0n}(z)=\partial_zX^y_{i-1j}(z)& \text{ if } \tau = \frac{1}{n} \text{ for some } n\in \N\\
           0=\partial_zX^y_{ij}(z)& \text{ else,}\end{cases}
        \end{align*}
        for all $i,j\ge 0$. The left equation implies using \cref{lemma: number game} that
        \begin{align*}
            X^y_{ij}(z)=0\begin{cases}
            \text{ for } (i,j)\notin \{ (1,0),(0,1)\}&\text{ if } \tau = 1,\\
            \text{ for }\ (i,j)\ne (0,1) &\text{ else,}
            \end{cases}
        \end{align*}
        Therefore, the right equation allows us to conclude that,
        \begin{itemize}
            \item if $\tau = 1$, then:
            \begin{align*}
            X^x_{ij}(z)=\begin{cases}
            X^x_{01}\in \R & \text{ if }(i,j)=(0,1)\\ 
            0 &\text{ else,}
            \end{cases}\qquad \text{ and } \qquad X^y_{ij}(z)=\begin{cases}
            X^y_{ij}\in \R & \text{ if }(i,j)\in \{(1,0),(0,1)\}\\ 
            0 &\text{ else,}
            \end{cases}
            \end{align*}
            \item if $\tau = \frac{1}{n}$ for $2\le n\in \N$, then:
            \begin{align*}
            X^x_{ij}(z)=\begin{cases}
            X^x_{0n}\in \R & \text{ if }(i,j)=(0,n)\\ 
            0 &\text{ else,}
            \end{cases}\qquad \text{ and } \qquad X^y_{ij}(z)=\begin{cases}
            X^y_{ij}\in \R & \text{ if }(i,j)= (0,1)\\ 
            0 &\text{ else,}
            \end{cases}
            \end{align*}
            \item else we have:
            \begin{align*}
            X^x_{ij}(z)=
            0 \ \text{ for all } i,j\ge 0, \qquad \text{ and } \qquad X^y_{ij}(z)=\begin{cases}
            X^y_{ij}\in \R & \text{ if }(i,j)=(0,1)\\ 
            0 &\text{ else,}
            \end{cases}
            \end{align*}
        \end{itemize}
    \end{itemize}
    
    \underline{In degree $2$:} For $W \in \mathfrak{X}_{\Rho _F}^2(\R^3)$, we have according to \eqref{eq: general 2} that 
    \begin{equation}
	    \mathrm{d}_{\pi_\tau} W = (1 + \tau-\lie_{E_\tau}) W^z + x\partial_zW^x+ \tau y\partial_z W^y )\partial_x\wedge\partial_y\wedge\partial_z.\label{eq: general 62}
    \end{equation}
    We first note that by \eqref{eq: general 611}, \eqref{eq: general 612} and \eqref{eq: general 613} we have
        \begin{align}
            W^x_{ij}(z)=&\, (i+\tau (j-1))X^y_{ij}(z)+\tau \partial_zX^z_{ij-1}(z)\label{eq: general 621b} \\
            W^y_{ij}(z)=&\, (i+\tau j-1)X^x_{ij}(z)+\partial_zX^z_{i-1j}(z) \label{eq: general 622b}\\
            W^z_{ij}(z)=&\, \partial_z X^y_{i-1j}(z) - \tau \partial_z X^x_{ij-1}(z)\label{eq: general 623b}
        \end{align}
    We distinguish between the different cases:
    \begin{itemize}
        \item if $\tau =-\frac{p}{q}$ with $p\le q \in \N$ relatively prime:
        
        By choosing:
        \begin{align*}
            X^y_{ij}(z) \ \text{ for } \ (i,j)\ne (np,nq+1) \qquad \text{ and } \qquad X^z_{ij}(z) \ \text{ for all } \ i,j 
        \end{align*}
        where $n\in \N_0$, we may assume, using \eqref{eq: general 621b} and \eqref{eq: general 622b}, that
        \begin{align*}
            W^x_{ij}(z)=0 \ \text{ for all } \ i,j \qquad \text{ and }\qquad W^y_{ij}(z)=0 \ \text{ for } \ (i,j)\ne (np+1,nq)  
        \end{align*}
        Moreover, equation \eqref{eq: general 623b} implies that by choosing $X^y_{np,nq+1}(z)$ we may assume that
        \begin{align*}
            W^z_{np+1,nq+1}(z)=0.
        \end{align*}
        Then the cocycle condition for $W$ is by \eqref{eq: general 62} equivalent to the equation
        \begin{align*}
            0=(1 -i -\tau(j-1))W^z_{ij}+\delta_{i(np+1)}\delta_{j(nq+1)}\tau\partial_zW^y_{np+1,nq}
        \end{align*}
        which implies that
        \begin{align*}
            W^z_{ij}(z)=0 \ \text{ for all }\ i,j \qquad \text{ and }\qquad W^y_{np+1,nq}(z)=W^y_{np+1,nq}\in \R .
        \end{align*}
        \item if $\tau =1 $, then:
        
        Equations \eqref{eq: general 621b} - \eqref{eq: general 623b} imply that we may assume that
        \begin{align*}
            W^x_{ij}(z)=0\ \text{ for }&\, \ (i,j)\notin \{ (1,0),(0,1)\}, \qquad W^y_{ij}(z)=0\ \text{ for } \ (i,j)\ne (0,1) \\
            \text{ and }&\, \qquad W^z_{ij}(z)=0\ \text{ for } \ (i,j)\in \{ (2,0),(1,1),(0,2)\}.
        \end{align*}
        The cocycle condition for $W$ is then by \eqref{eq: general 62} equivalent to
        \begin{align*}
            0=(2 -i -j))W^z_{ij}(z)\ \text{  for }\ i+j\ne 2 \qquad \text{ and } \qquad 0=\partial_z W^x_{i-1j}(z)+\delta_{i0}\delta_{j2}\partial_zW^y_{01}(z) \ \text{ for } \ i+j=2
        \end{align*}
        Hence we obtain that:
        \begin{align*}
            W^z_{ij}(z)= 0 \ \text{ for all } \ i,j, \qquad \text{ and }\qquad W^x_{10}(z)=W^x_{10},\, W^x_{01}(z)=W^x_{01},\, W^y_{01}(z)=W^y_{01} \in \R
        \end{align*}
        \item if $\tau=\frac{1}{n}$ for $2\le n\in\N$, then:
        
        Similar as above we obtain here from \eqref{eq: general 621b} - \eqref{eq: general 623b} imply that we may assume that
        \begin{align*}
            W^x_{ij}(z)=0\ \text{ for }&\, \ (i,j)\notin \{ (0,1)\}, \qquad W^y_{ij}(z)=0\ \text{ for } \ (i,j)\ne (0,n) \\
            \text{ and }&\, \qquad W^z_{ij}(z)=0\ \text{ for } \ (i,j)\in \{ (1,1),(0,n+1)\}.
        \end{align*}
        The cocycle condition for $W$ is then by \eqref{eq: general 62} equivalent to
        \begin{align*}
            0=(1 -i -\tau(j-1)))W^z_{ij}(z)+\delta_{i1}\delta_{j1}\partial_z W^x_{i-1j}(z)+\delta_{i0}\delta_{jn+1}\tau \partial_zW^y_{01}(z) 
        \end{align*}
        Hence we obtain that:
        \begin{align*}
            W^z_{ij}(z)= 0 \ \text{ for all } \ i,j, \qquad W^x_{01}(z)=W^x_{01},\, W^y_{0n}(z)=W^y_{0n} \in \R .
        \end{align*}
        \item In all the other cases we obtain the following:
        
        First \eqref{eq: general 621b}, \eqref{eq: general 622b} and \eqref{eq: general 623b} imply that we may assume that
        \begin{align*}
            W^x_{ij}(z)=0\ \text{ for }&\, \ (i,j)\notin \{ (0,1)\}, \qquad W^y_{ij}(z)=0\ \text{ for all } \ i,j \\
            \text{ and }&\, \qquad W^z_{ij}(z)=0\ \text{ for } \ (i,j)\in \{ (1,1)\}.
        \end{align*}
        The cocycle condition for $W$ is again by \eqref{eq: general 62} equivalent to
        \begin{align*}
            0=(1 -i -\tau(j-1)))W^z_{ij}(z)+\delta_{i1}\delta_{j1}\partial_z W^x_{i-1j}(z)
        \end{align*}
        Hence we obtain that:
        \begin{align*}
            W^z_{ij}(z)= 0 \ \text{ for all } \ i,j, \qquad \text{ and } \qquad W^x_{01}(z)=W^x_{01} \in \R .
        \end{align*}
    \end{itemize}
    
    \underline{In degree $3$:} From \eqref{eq: general 62} we obtain for a general coboundary element the equation
    \begin{align*}
        h_{ij}(\pi)(z)=(1 -i -\tau(j-1)))W^z_{ij}(z)+\partial_z W^x_{i-1j}(z)+\tau\partial_z W^y_{ij-1}(z)
    \end{align*}
    which is clearly surjective onto $h_{ij}(z)$ for all $i,j$.
    
    \underline{Algebraic structure:} All the statements are implied by \eqref{eq: general 621b} - \eqref{eq: general 623b}.
\end{proof}

\subsection{Flat Poisson cohomology for the open book-type}
The goal of this section is to compute $H^{\bullet}_{\Rho _F}(\R^3,\pi_\tau)$. The idea for $\pi_1$ is to calculate the flat cohomology using cylindrical coordinates, which induces a Poisson diffeomorphism
\begin{align}\label{eq:cylindrical_coord_model}
    \varphi: (M = (0,\infty) \times S^1 \times \R,\pi_M:=r\partial_r\wedge \partial_s )&\, \to (N:=\R^3\backslash \Rho ,\pi_N:= \pi_1|_N)\\
    (r, \theta, s) \qquad \qquad \qquad &\,\mapsto \qquad (r\cos \theta, r\sin \theta, s)\nonumber 
\end{align}
and hence an isomorphism $H^\bullet(M, \pi_M) \cong H^\bullet(N, \pi_N)$. While the change of coordinates is only smooth on $\R^3\backslash P$ it actually induces an isomorphism between the rng of $P$-flat functions on $\R^3$ and the functions which are flat at $\{r=0\}\subset [0,\infty)\times S^1\times \R$ and hence we obtain an isomorphism in \enquote{flat cohomology}. We adapt this idea to the cases $\mathfrak{b}_\tau$ with $0<\tau\le 1$.

\subsubsection{Elliptic cylindrical coordinates}

Let $0 < \tau \le 1$ and define $\varphi_\tau: M \rightarrow N$ by 
\begin{equation}\label{eq:psi_diffeo}
	(r,\theta,z) \mapsto (r\cos \theta, r^\tau \sin \theta, z). 
\end{equation}
See \cref{fig:graphi_psi} for an image of this coordinate transformation. Note that this map is smooth since $r > 0$. 

\begin{wrapfigure}{l}{5cm}
\vspace{-10pt}
\centering
\scalebox{1.5}{
\begin{tikzpicture}
    \node[above right,inner sep=0] (image) at (0,0) {
	\includegraphics[width=0.5\linewidth]{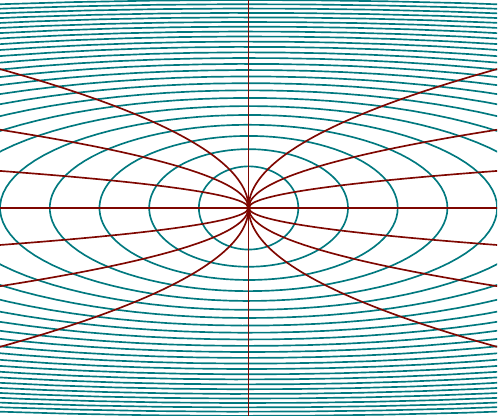}
    };
\begin{scope}[
x={($0.1*(image.south east)$)},
y={($0.1*(image.north west)$)}]
    \draw[->] (0,5) -- node[pos=1,above] {\scriptsize $x$} (10.3,5);
    \draw[->] (5,0) -- node[pos=1,right] {\scriptsize $y$} (5,10.3);
\end{scope}
\end{tikzpicture}}
    \caption{The map \cref{eq:psi_diffeo} with images for constant $\textcolor{blue}{r}$ and  $\textcolor{red}{\theta}$.}\label{fig:graphi_psi}
    \vspace{-10pt}
\end{wrapfigure}
\noindent First we compute for the differential:
\begin{equation}\label{eq: differential positive transformation}
	\diff\varphi_{\tau} = \begin{pmatrix}
		\cos \theta & -r\sin \theta & 0 \\
		\tau r^{\tau - 1} \sin \theta & r^\tau \cos \theta & 0 \\
		0 & 0 & 1
	\end{pmatrix}
\end{equation}
Note that the differential is invertible for all points in $M$ since
\begin{equation*}
	 \det \diff \varphi_{\tau}(r,\theta,s)  = r^\tau \cos^2 \theta + \tau r^\tau \sin^2 \theta \geq \tau r^\tau > 0.
\end{equation*}
Therefore, $\varphi_{\tau}$ is
a diffeomorphism iff it is bijective by the inverse function theorem. Assuming that this is the case we obtain
\begin{equation}\label{eq: poisson diffeo positive tau}
    (\varphi_{\tau})_\ast \pi_M = (r\cos\theta \partial_x + \tau r^\tau \sin \theta) \wedge \partial_s = (x\partial_x + \tau y\partial_y) \wedge \partial_s = \pi_\tau.
\end{equation}

Therefore it remains to prove that $\varphi_{\tau}$ is bijective. For $\tau =1$ this is well-known, hence we assume $0<\tau <1$. We begin this task with some auxiliary lemmas.
\begin{lemma}\label{lemma: bijective}
	Let $I, J \subseteq \R$ be two open intervals and set $a := \inf I, b := \sup I$. Let $f: I \rightarrow J$ be a continuous, strictly monotone function. Then:
	\begin{equation*}
		\Set{\lim_{t \searrow a} f(t), \lim_{t \nearrow b} f(t)} = \Set{\inf J, \sup J} \qquad \Rightarrow \qquad f \text{ is bijective.} 
	\end{equation*}
\end{lemma}
\begin{proof}
	Since $f$ is strictly monotone, it is injective. Assume without loss of generality that $f$ is strictly increasing, else we multiply by $-1$. Then $\lim_{t \searrow a} f(t) = \inf f(I)$ and $\lim_{t \nearrow b} f(t) = \sup f(I)$. For $c \in J$ there exists $\varepsilon>0$ such that $\inf J + \varepsilon < c < \sup J - \varepsilon$ and $x, y \in I$ such that 
	\begin{align*}
	    f(x) < \inf J + \varepsilon < c < \sup J - \varepsilon < f(y).
	\end{align*}
	Therefore, there exists $z \in (x, y)$ such that $f(z) = c$ by the intermediate value theorem and hence $f$ is surjective.
\end{proof}

With the previous lemma we can prove the following result.

\begin{lemma}\label{lemma:angle_inverse}
	Let $0 < \tau < 1$. The maps $f: \left(0, \frac{\pi}{2}\right) \rightarrow (0, \infty)$ and $g: \left(\frac{\pi}{2},\pi\right) \rightarrow (0, \infty)$ given by
	\begin{equation*}
		f(t):=\frac{(\cos t)^\tau}{\sin t}
	\qquad \text{ and }\qquad g(t):=\frac{(-\cos t)^\tau}{\sin t}\qquad \text{ are bijections.}
	\end{equation*}
\end{lemma}
\begin{proof}
	For $0 < t < \frac{\pi}{2}$ the sine and cosine are strictly positive, hence $f$ is well-defined and continuous. The function $f$ is strictly monotone iff $\log f$ is strictly monotone. The strict monotonicity of $\log f$ follows from
	\begin{align*}
	    \frac{\mathrm{d}(\log f(t))}{\dt} &= \frac{\mathrm{d}}{\dt}\left(\tau \log \cos t - \log \sin t\right) = -(\tau \tan t + \cot t)  \leq -\frac{\tau}{\sin t \cos t} < 0.
	\end{align*}
	Finally, as $t \searrow 0$ we have $\cos t \rightarrow 1$ and $\sin t \rightarrow 0$, thus $\lim_{t \searrow 0} f(t) = \infty$. Similarly, as $t \nearrow \frac{\pi}{2}$ we have $\cos t \rightarrow 0$ and $\sin t \rightarrow 1$, thus $\lim_{t \nearrow \frac{\pi}{2}} f(t) = 0$. Therefore \cref{lemma: bijective} implies that $f$ is bijective. The proof for $g$ follows along the same lines.
\end{proof}

We have the following useful identities. For $x, y > 0$ we have
\begin{equation}\label{eq:fraction_f}
	\frac{\left(\cos f^{-1}\left(x^\tau / y\right)\right)^\tau}{\sin f^{-1}\left(x^\tau / y\right)} = f\left(f^{-1}\left(\frac{x^\tau}{y}\right)\right) = \frac{x^\tau}{y}.
\end{equation}
For $x < 0$ and $y > 0$ we have
\begin{equation}\label{eq:fraction_g}
	\frac{\left(-\cos g^{-1}\left((-x)^\tau / y\right)\right)^\tau}{\sin g^{-1}\left((-x)^\tau / y\right)} = g\left(g^{-1}\left(\frac{(-x)^\tau}{y}\right)\right) = \frac{(-x)^\tau}{y}.
\end{equation}
The bijectivity of $\varphi_{\tau}$ follows from the following lemma.
\begin{lemma}
	Let $0 < \tau < 1$. The map $\chi_\tau: (0,\infty) \times S^1 \rightarrow \R^2 \backslash \Set{0}$ given by
	\begin{equation*}
		(r,\theta) \mapsto (r\cos\theta, r^\tau \sin \theta)
	\end{equation*}
	is bijective.
\end{lemma}
\begin{proof}
	In this proof we think of $S^1$ as $[0, 2\pi)$. We define the sets
	\begin{align*}
		Q_1 &= \Set{(x,y) \in \R^2}[x,y > 0]& Q_2 &= \Set{(x,y) \in \R^2}[x < 0 < y]&		Q_3 &= -Q_1 & Q_4 &= -Q_2\\
		P_x &= \Set{(x,0) \in \R^2}[x > 0] & P_y &= \Set{(0,y) \in \R^2}[y > 0] & N_x &= -P_x & N_y &= -P_y.
	\end{align*}
	and note that
	\begin{equation*}
		\R^2 \backslash \Set{0} = Q_1 \sqcup Q_2 \sqcup Q_3 \sqcup Q_4 \sqcup P_x \sqcup P_y \sqcup N_x \sqcup N_y.
	\end{equation*}
	
	It is easy to see that for $\theta_0= 0,\frac{\pi}{2},\pi,\frac{3\pi}{2}$ respectively, $\chi_\tau$ is bijective from $(0,\infty)\times \{\theta_0\}$ onto $P_x,P_y,N_x,N_y$ respectively.

	For $Q_1$ we define
	\begin{equation*}
		F_{1,\tau}: Q_1 \rightarrow (0, \infty) \times \left(0, \frac{\pi}{2}\right) \qquad (x,y) \mapsto \left(\frac{x}{\cos f^{-1}(x^\tau / y)}, f^{-1}(x^\tau / y)\right)
	\end{equation*}
	This function is well-defined since for $(x,y) \in Q_1$ we have $x, y > 0$ and therefore $\frac{x^\tau}{y} > 0$.
	We have
	\begin{equation*}
		F_{1,\tau}(\chi_{\tau}(r, \theta)) = \left(\frac{r\cos\theta}{\cos f^{-1}(f(\theta))}, f^{-1}(f(\theta))\right) = (r,\theta)
	\end{equation*}
	and
	\begin{equation*}
		\chi_{\tau}(F_{1,\tau}(x, y)) = \chi\left(\frac{x}{\cos f^{-1}(x^\tau / y)}, f^{-1}(x^\tau / y)\right) = \left(x, \frac{x^\tau\sin f^{-1}(x^\tau / y)}{(\cos f^{-1}(x^\tau / y))^\tau}\right) = (x,y)
	\end{equation*}
	by \cref{eq:fraction_f}. Thus $\chi_{\tau}|_{(0, \infty) \times \left(0, \frac{\pi}{2}\right)}$ is bijective onto $Q_1$. Similarly, $F_{3,\tau}(x,y):=F_{1,\tau}(-(x,y))+(0,\pi)$ implies that $\chi_{\tau}|_{(0, \infty) \times \left(\pi, \frac{3\pi}{2}\right)}$ is bijective onto $Q_3$. 

	On $Q_2$ the map 
	\begin{equation*}
		F_2: Q_2 \rightarrow (0, \infty) \times \left(\frac{\pi}{2}, \pi\right), \qquad (x,y) \mapsto \left(\frac{x}{\cos g^{-1}((-x)^\tau / y)}, g^{-1}((-x)^\tau / y)\right)
	\end{equation*}
	defines an inverse of $\chi_{\tau}$ which can be seen using \cref{eq:fraction_g}.	Thus $\chi_{\tau}|_{(0, \infty) \times \left(\frac{\pi}{2}, \pi\right)}$ is bijective onto $Q_2$. Finally, $F_{4,\tau}(x,y):=F_{2,\tau}(-(x,y))+(0,\pi)$ implies that $\chi_{\tau}|_{(0, \infty) \times \left(\frac{3\pi}{2}, 2\pi\right)}$ is bijective onto $Q_4$.
\end{proof}

\begin{corollary}
The map $\varphi_{\tau}$ is bijective for $0<\tau \le 1$.
\end{corollary}

\subsubsection{An isomorphism of rngs and modules}
The main purpose of this subsection is to prove the following:
\begin{proposition}\label{proposition: positive tau isomorphism}
The map $\varphi_{\tau}$ induces an isomorphism of rng's:
\begin{align}\label{eq: flat isomorphism positive}
    \varphi_{\tau}^*:C^{\infty}_{\Rho_f}(\R^3)\to C^{\infty}_{R_f}([0,\infty)\times S^1\times \R)\qquad \text{ where }\ R:=\{0\}\times S^1\times \R =\partial M \subset [0,\infty)\times S^1\times \R
\end{align}
and $C^{\infty}_{R_f}([0,\infty)\times S^1\times \R)$ denotes the space of smooth functions which are flat at $R$.
\end{proposition}

We separate the proof of this statement into two lemma. Let us first show that this map is well-defined.
\begin{lemma}\label{lemma:cylindrical vanishing}
    Let $f\in C^\infty_{\Rho _f}(\R^3)$ and consider $f^{\varphi_{\tau}} := \varphi_{\tau}^*f$.
    Then for all $\alpha=(\alpha_1,\alpha_2,\alpha_3)\in \N_0^3$ we have
    \begin{equation*}
	\partial^{\alpha}_{(r,\theta,s)} f^\varphi (r,\theta,s) = \sum_{\substack{\alpha \in \N_0 \\ \beta \le \alpha_1+\alpha_2}} P_{\beta}(r^\tau,r,r^{-1},\cos \theta ,\sin \theta) ( \partial_{(x,y,z)}^{(\beta,\alpha_3)} f)(r\cos \theta, r\sin \theta, s).
    \end{equation*}
    where $P_{\beta}$ is a polynomial of total degree at most $3\alpha_1+2\alpha_2$.
    
    In particular, the map \eqref{eq: flat isomorphism positive} is a well-defined injective map of rng's.
\end{lemma}
\begin{proof}
    For $|\alpha| = 0$ this claim is immediate. Suppose $|\alpha| \geq 0$ and that the claim holds for $|\alpha|$, we show that it holds for any $\tilde{\alpha}\in \N_0^3$ with $|\tilde{\alpha}|=|\alpha|+1$. For this note that for a polynomial $P$ of degree $n$ we have
    \begin{align*}
	\partial_r P(r^\tau,r,r^{-1},\cos \theta ,\sin \theta) = \tilde{P}_r(r^\tau,r,r^{-1},\cos \theta ,\sin \theta), \qquad &\, \partial_{\theta}P(r^\tau,r,r^{-1},\cos \theta ,\sin \theta)=\tilde{P}_{\theta}(r^\tau,r,r^{-1},\cos \theta ,\sin \theta)
    \end{align*}
    where $\tilde{P}_r$ is of degree $n+1$ and $\tilde{P}_{\theta}$ is of degree $n$. Moreover, due to \eqref{eq: differential positive transformation} every time $\partial_r$ is applied to $f$ the degree of the polynomial increases by at most $3$ and similar for $\partial_\theta $ by at most $2$, which implies the claim about the derivatives. 
    
    To show that \eqref{eq: flat isomorphism positive} is well defined, note first that we only need to check what happens for the limit $r\to 0$. Recall that by Taylor expansion we obtain for any $h\in C^\infty_{\Rho _f}(\R^3)$ around $\Rho$ for any $k\in \N$ that
    \begin{align*}
        h(x,y,z)=\sum_{\substack{\gamma\in \N^2_0\\|\gamma|=k}}x^{\gamma_1}y^{\gamma_2}h_{\gamma}(x,y,z)
    \end{align*}
    where all $h_{\gamma}\in C^\infty_{\Rho _f}(\R^3)$. Hence the statement follows by applying this expansions to the various derivatives of $f$ above for $n=3|\alpha|$ and taking the limit $r\to 0$.
\end{proof}
\cref{proposition: positive tau isomorphism} the follows from the following lemma.
\begin{lemma}\label{lemma:cylindrical vanishing}
    Let $\psi_{\tau}:=\varphi^{-1}_{\tau}$ where this is well defined. For $g\in C^{\infty}_{R_f}([0,\infty)\times S^1\times \R)$ consider $g^{\psi_{\tau}} := \psi_{\tau}^*g$.
    Then for all $\alpha=(\alpha_1,\alpha_2,\alpha_3)\in \N_0^3$ we have
    \begin{equation*}
	\partial^{\alpha}_{(r,\theta,s)} g^{\psi_{\tau}} (r,\theta,s) = \sum_{\substack{\alpha \in \N_0 \\ \beta \le \alpha_1+\alpha_2}} \frac{P_{\beta}(r^\tau,r,r^{-1},\cos \theta ,\sin \theta)}{(r^\tau \cos^2 \theta + \tau r^\tau \sin^2\theta)^{\alpha_1+\alpha_2}} ( \partial_{(r,\theta,s)}^{(\beta,\alpha_3)} g)(\psi_{\tau}(x,y,z)).
    \end{equation*}
    where $P_{\beta}$ is a polynomial of total degree at most $4\alpha_1+3\alpha_2$.
    
    In particular, the map \eqref{eq: flat isomorphism positive} is surjective.
\end{lemma}
\begin{proof}
    We first note that the differential of $\psi_{\tau}$ is due to \eqref{eq: differential positive transformation} given by
    \begin{equation*}
	\diff\psi_{\tau}(x,y,z) = \begin{pmatrix}
		\frac{r^\tau \cos \theta}{r^\tau \cos^2 \theta + \tau r^\tau \sin^2\theta} &  \frac{r\sin \theta}{r^\tau \cos^2 \theta + \tau r^\tau \sin^2\theta}& 0 \\
		\frac{-\tau r^{\tau - 1} \sin \theta }{r^\tau \cos^2 \theta + \tau r^\tau \sin^2\theta}&\frac{ \cos \theta }{r^\tau \cos^2 \theta + \tau r^\tau \sin^2\theta}& 0 \\
		0 & 0 & 1
	\end{pmatrix}(\psi_{\tau}(x,y,z))
\end{equation*}
    Due to this expression, every time $\partial_x$ is applied to $g$ the degree of the polynomial increases by at most $3$ and the divisor degree by $1$ and similar for $\partial_y $ by at most $2$ and the divisor by degree $1$. The change of a polynomial $P_{\beta}$ under $\partial_r$ and $\partial_\theta$ is described in the proof of the previous lemma. Moreover, we note that
    \begin{align*}
        \partial_r \frac{1}{(r^\tau \cos^2 \theta + \tau r^\tau \sin^2\theta)^n}=&\, \frac{-n\tau r^{\tau-1}( \cos^2 \theta + \tau \sin^2\theta)}{(r^\tau \cos^2 \theta + \tau r^\tau \sin^2\theta)^{n+1}}\\ \partial_\theta \frac{1}{(r^\tau \cos^2 \theta + \tau r^\tau \sin^2\theta)^n}=&\, \frac{-2n\cos \theta\sin \theta r^{\tau}(\tau-1) }{(r^\tau \cos^2 \theta + \tau r^\tau \sin^2\theta)^{n+1}}
    \end{align*}
    Hence the claim about the derivatives follows.

    To show that \eqref{eq: flat isomorphism positive} is surjective, note first that we only need to check what happens for limits $(x_i,y_i,z_i)\to (0,0,z_0)$ for any $z_0\in \R$. Taylor expansion implies for any $h\in C^{\infty}_{R_f}([0,\infty)\times S^1\times \R)$ around $R$ for any $k\in \N$ that
    \begin{align*}
        h(r,\theta,s)=r^{k}h_{k}(r,\theta,s)
    \end{align*}
    where $h_{k}\in C^{\infty}_{R_f}([0,\infty)\times S^1\times \R)$. Hence the statement follows using $k=5|\alpha|$ since for $(x_i,y_i,z_i)\to (0,0,z_0)$ we have $(r_i,\theta_i,s_i)\to (0,\theta_0,s_0)$.
\end{proof}
Combining these two lemmas we obtain \cref{proposition: positive tau isomorphism}. An immediate consequence is the following:
\begin{corollary}\label{corollary: flat isomorphism positive tau}
The map $\varphi_{\tau}$ induces an isomorphism:
\begin{align*}
    \varphi_{\tau}^*:\mathfrak{X}^{\bullet}_{\Rho_f}(\R^3)\to \mathfrak{X}^{\bullet}_{R_f}([0,\infty)\times S^1\times \R)
\end{align*}
\end{corollary}
\begin{proof}
We note that we have
\begin{align*}
    (\varphi_{\tau})_*(\partial_r)=&\,  \cos\theta \partial_x+\tau r^{\tau -1}\sin \theta\partial_y,\qquad (\varphi_{\tau})_*(\partial_\theta)= -r\sin\theta \partial_x+ r^{\tau}\cos \theta\partial_y, \qquad (\varphi_{\tau})_*(\partial_s)= \partial_z\\
    &\, (\psi_{\tau})_*(\partial_x)= \frac{r^\tau \cos \theta}{r^\tau \cos^2 \theta + \tau r^\tau \sin^2\theta}\partial_r-\frac{\tau r^{\tau-1} \sin \theta}{r^\tau \cos^2 \theta + \tau r^\tau \sin^2\theta}\partial_\theta,\\ (\psi_{\tau})_*&\, (\partial_y)= \frac{r \sin \theta}{r^\tau \cos^2 \theta + \tau r^\tau \sin^2\theta} \partial_r+ \frac{\cos \theta}{r^\tau \cos^2 \theta + \tau r^\tau \sin^2\theta}\partial_\theta, \qquad (\psi_{\tau})_*(\partial_z)= \partial_s
\end{align*}
Since these sections define a trivialization for the tangent bundle of both manifolds, hence flat vector fields are uniquely given by coefficients in flat functions for each generator. Hence the statement follows by using the Taylor expansion for flat functions as in the proofs of the previous two lemmas.   
\end{proof}

\subsubsection{Conclusion}

\begin{proposition}\label{prop:flat_cylindrical}
    The flat Poisson cohomology groups $H^{\bullet}_{R_r}([0,\infty)\times S^1\times \R, r\partial_r\wedge \partial_s)$ are trivial.
\end{proposition}
\begin{proof}
Note the proof follows along the exact same lines as the proof of \cref{thm:affine3} in \cref{sec:affine} with $x\to r$, $y\to z$ and $z\to \theta$. Hence we obtain representatives described by functions in $\theta$ which are flat at $\{r=0\}$. The flatness together with the independence of $r$ implies that the functions have to be zero and hence all cohomology groups are trivial.
\end{proof}

\begin{corollary}\label{cor:flat_cohomology_btau}
	Let $0 < \tau \leq 1$, then the groups $H^{\bullet}_{\Rho_f}(\R^3, \pi_\tau) $ are trivial.
\end{corollary}
\begin{proof}
    By \cref{corollary: flat isomorphism positive tau} and \eqref{eq: poisson diffeo positive tau} the map \eqref{eq:psi_diffeo} induces an isomorpism of complexes
    \begin{align*}
       ( \mathfrak{X}^{\bullet}_{\Rho_f}(\R^3),\diff_{\pi_{\tau}})\simeq (\mathfrak{X}^{\bullet}_{R_f}([0,\infty)\times S^1\times \R),r\partial_r\wedge\partial_s)
    \end{align*}
    The result now follows from \cref{prop:flat_cylindrical}.
\end{proof}

This yields
\begin{proof}[Proof of \cref{thm:open_book t1}]
	From \cref{cor:flat_cohomology_btau} we see that $H^{\bullet}_{\Rho_f}(\R^3 , \pi_\tau)$ is trivial. Therefore the corresponding long exact sequence \eqref{eq: les R} induces an isomorphism
	\begin{align*}
	    H^{\bullet}(\R^3 , \pi_\tau)\simeq H^{\bullet}_{\Rho_F}(\R^3, \pi_\tau)
	\end{align*}  
	Hence the statement follows from \cref{thm:open_book_formal}.
\end{proof}

\subsection{Flat Poisson cohomology for the hyperbolic-type}
Now we consider the Poisson structure $(\R^3, \pi_\tau )$ for $-1 \le \tau < 0$. To treat these Poisson structures we use the remaining short exact sequences in \cref{proposition: ses}, i.e. the $\Rho$-flat Poisson cohomology $H^{\bullet}_{\Rho_f}(\R^3, \pi_\tau) $ fits into a long exact sequence
\begin{equation}\label{eq: x-flat}
\ldots \stackrel{j^{\infty}_Y}{\to} H^{q-1}_{Y_F}(\R^3, \pi_\tau)\stackrel{\partial}{\to} 
H^{q}_{Y_f}(\R^3, \pi_\tau)\to H^{q}_{\Rho_f}(\R^3, \pi_\tau)\stackrel{j^{\infty}_Y}{\to}H^{q}_{Y_F}(\R^3, \pi_\tau )\stackrel{\partial}{\to}\ldots. 
\end{equation}
and the $Y$-flat Poisson cohomology $H^{\bullet}_{Y_f}(\mathfrak{b}_\tau^\ast, \pi_\tau)$
fits in the long exact sequence
\begin{equation}\label{eq: y-flat}
\ldots \stackrel{j^{\infty}_{X}}{\to} H^{q-1}_{XY_F}(\R^3, \pi_\tau)\stackrel{\partial}{\to} H^{q}_{XY_f}(\R^3, \pi_\tau)\to H^{q}_{Y_f}(\R^3, \pi_\tau)\stackrel{j^{\infty}_{X}}{\to}H^{q}_{XY_F}(\R^3, \pi_\tau )\stackrel{\partial}{\to}\ldots. 
\end{equation}
We first compute the two formal cohomology groups in the next section and thereafter the flat ones.

\subsubsection{Formal along the $x$ and $y$-axis}
In this section we compute Poisson cohomology with coefficients in $C_{U_f}^\infty(\R^2)[[y]]$ and $C_{U_f}^\infty(\R^2)[[x]]$ respectively, i.e. in the set of formal power series in $y$ with coefficients being smooth functions in $(x,z)$ which vanish flatly along $x=0$ and in $x$ with coefficients being smooth functions in $(y,z)$ which vanish flatly along $y=0$, respectively.
\begin{proposition}\label{prop: formal along y and xy}
For $-1\le \tau <0$ the Poisson cohomology groups $H^{\bullet}_{Y_F}(\R^3,\pi_{\tau})$ and $H^{\bullet}_{XY_F}(\R^3,\pi_{\tau})$ are trivial.
\end{proposition}
For the prove the following lemma will be useful.
\begin{lemma}\label{lemma: invertable}
The operator $(i+y\partial_y):C_y^\infty(\R^2)\to C_y^\infty(\R^2)$ is a bijection for all $i\in \R$.
\end{lemma}
\begin{proof}
Note that for $i=0$ its inverse is given by
\begin{align*}
    \int_{0}^y\frac{\cdot}{\tilde{y}}\diff   \tilde{y}
\end{align*}
For $0\ne i\in \R$ we have that:
\begin{align*}
    (i+y\partial_y)(-\frac{1}{y^i}\partial_y)=-\frac{1}{y^{i-1}}\partial_y^2
\end{align*}
and $\frac{1}{y^i}\partial_y$ and $\frac{1}{y^{i-1}}\partial_y^2$ are invertable on $C_y^\infty(\R^2)$.
\end{proof}

\begin{proof}
We first note that the proofs for the coefficients in $C_x^\infty(\R^2)[[y]]$ and $C_y^\infty(\R^2)[[x]]$ follow along the same lines, exchanging the roles of $x$ and $y$. We write the case $C_y^\infty(\R^2)[[x]]$.

\underline{In degree $0$ :} \  The differential of $g\in C_y^\infty(\R^2)[[x]]$ is given by
\begin{equation*}
    \mathrm{d}_{\pi_\tau} g = x\partial_z g\partial_x + \tau y \partial_z g\partial_y -(x\partial_x + \tau y\partial_y)g\partial_z
\end{equation*}
If we write $g = \sum_{i = 0}^\infty g_{i}(y,z)x^i$, then $g$ being a cocycle implies for all $i\ge 0$ that
\begin{align*}
    i g_{i}(y,z)+\tau y\partial_y g_i(y,z) = 0 \qquad \text{ and } \qquad \tau y\partial_z g_{i}(y,z) = 0.
\end{align*}
The first equation together with \cref{lemma: invertable} implies that $g_i=0$ for all $i$.

\underline{In degree $1$ :} \ Let $X\in \mathfrak{X}^1_{y}(\R^2)[[x]]$. Note that by a choice of $g$ and using \cref{lemma: invertable} we may assume that $X^z=0$. Equations\eqref{eq: general 611} and \eqref{eq: general 612} then yield
\begin{align}
    W^x(\pi_{\tau})=\, \lie_{E_\tau} X^y - \tau X^y\qquad \text{ and }\qquad   W^y(\pi_{\tau})=\, -\lie _{E_\tau} X^x + X^x.\label{eq: semiformal:111}
\end{align}
Hence \cref{lemma: invertable} implies that $X^x=X^y=0$ for a Poisson vector field.

\underline{In degree $2$ :} \ Let $W \in \mathfrak{X}_y^2(\R^2)[[x]]$, then by \eqref{eq: semiformal:111} and \cref{lemma: invertable} we may assume that $W^x=W^y=0$. By \eqref{eq: general 62} the Poisson differential then becomes
    \begin{equation*}
	    \mathrm{d}_{\pi_\tau} W = -(\lie_{E_\tau} W^{z} - (1 + \tau)W^{z})\partial_x\wedge\partial_y\wedge\partial_z.
    \end{equation*}
Hence once more \cref{lemma: invertable} implies that $W$ being Poisson implies $W^z=0$.
    
\underline{In degree $3$ :} \ The argument from degree $2$ implies that every element in $\mathfrak{X}_y^2(\R^2)[[x]]$ is a coboundary. 
\end{proof}

\subsubsection{The flat Poisson cohomology}
It remains to compute the cohomology groups $H^{q}_{XY_f}(\mathfrak{b}_\tau^\ast, \pi_\tau)$. To obtain this result we proceed similar as in case $\tau>0$, making use of a (Poisson) diffeomorphism away from $XY$ with controlled singularities at $XY$ which induces rng and module isomorphisms.

We use the following change of coordinates
\begin{equation}
    \begin{array}{cccc}
         \varphi_{ij,\tau}:&(0,\infty)^2\times\R &\to& D_{ij}  \\
         &(t,r,s)&\mapsto &((-1)^ie^rt,(-1)^je^{\tau r}t,s) 
    \end{array}
\end{equation}
where $i,j\in \{0,1\}$ and $D_{ij}$ is given by
\begin{align*}
    D_{ij}:=\{(x,y,z)\in \R^3|0<(-1)^ix \ \& \ 0<(-1)^jy\}.
\end{align*}
It is straightforward to check that these functions are actually bijections. For the differential we obtain
\begin{align*}
    \diff   \varphi_{ij,\tau} =\begin{pmatrix}
    (-1)^ie^r&(-1)^ie^r t&0\\
    (-1)^je^{\tau r}&(-1)^j\tau e^{\tau r}t&0\\
    0&0&1
    \end{pmatrix}
\end{align*}
and hence we have
\begin{align*}
    \det(\diff   \varphi_{ij,\tau})=(-1)^{i+j}(1-\tau)e^{\tau +1}\ne 0
\end{align*}
which implies that $\varphi_{ij}$ is indeed a diffeomorphism for all $i,j$. Moreover, these maps define Poisson diffeomorphism in the four quadrants, i.e.
\begin{align*}
    (\varphi_{ij,\tau})_* (\partial_r\wedge \partial_s)=\pi_{\tau}|_{D_{ij}}
\end{align*}
Similar as to case $\tau>0$ we obtain the following:
\begin{proposition}\label{proposition: isomorphism negative tau}
For $-1\le \tau <0$ the maps $\varphi_{ij,\tau}$ induce an isomorphism of rng's 
\begin{equation*}
(\varphi_{00,\tau}^*,\varphi_{01,\tau}^*,\varphi_{10,\tau}^*,\varphi_{11,\tau}^*):C^{\infty}_{XY_f}(\R^3)\to C^{\infty}_{V_f}([0,\infty)\times \R^2)^4
\end{equation*}
where $V:= \{0\}\times  \R^2 \subset [0,\infty) \times \R^2 $. Moreover the $\varphi_{ij,\tau}$ induces isomorphisms between the corresponding modules of flat multivectorfields.
\end{proposition}
As an immediate consequence we obtain that
\begin{corollary}
For the Poisson cohomology groups $H^{\bullet}_{XY_f}(\R^3,\pi_{\tau})$ we have the following unique representatives for cohomology classes:
\begin{itemize}
    \setlength\itemsep{0em}
	\item The Casimir functions are given by
		\begin{equation*}
			g(x,y,z):= \begin{cases} g_{00}(|x|^{-\tau} |y|) & if \ 0<x,y\\  g_{01}(|x|^{-\tau} |y|) & if \ 0<x,-y\\ g_{10}(|x|^{-\tau} |y|) & if \ 0<-x,y\\ g_{11}(|x|^{-\tau} |y|) & if \ 0<-x,-y \end{cases}
		\end{equation*}
where $g_{ij}\in C^{\infty}_{0_f}([0,\infty))$
	\item The Poisson vector fields form a free module over the Casimirs with generator
	\begin{equation*}
	    \partial_z
	\end{equation*}
	\item The other cohomology groups are trivial.
\end{itemize}
\end{corollary}
\begin{proof}
By \cref{proposition: isomorphism negative tau} the maps $\varphi_{ij,\tau}$ induces an ismorphism of Poisson complexes
\begin{equation*}
(\mathfrak{X}^{\bullet}_{XY_f}(\R^3),\diff_{\pi_\tau})\diffto (\mathfrak{X}^{\bullet}_{V_f}([0,\infty)\times \R^2)^4,\diff_{\oplus_{i=1}^4 r\partial_r\wedge\partial_s})
\end{equation*}
For the computations of the corresponding cohomology groups on $((0,\infty)\times [0,\infty)\times \R)^4$ we now follow the computations from \cref{sec:affine} under the given flatness assumptions.
\end{proof}
We conclude this section by proving \cref{thm:hyperbolic}.
\begin{proof}[Proof of \cref{thm:hyperbolic}]
Since the Poisson cohomology groups $H^{\bullet}_{Y_F}(\R^3,\pi_\tau)$ and $H^{\bullet}_{XY_F}(\R^3,\pi_\tau)$ are trivial by \cref{prop: formal along y and xy}, we obtain from the sequences \eqref{eq: x-flat} and \eqref{eq: y-flat} an isomorphism
\begin{equation}\label{eq: ismorphism xy-flat z-flat}
    H^{\bullet}_{XY_f}(\R^3,\pi_\tau)\simeq H^{\bullet}_{\Rho_f}(\R^3,\pi_\tau)
\end{equation}
Moreover, \cref{lemma: Borel} implies that the map
\begin{equation*}
    H^{\bullet}(\R^3,\pi_\tau)\xrightarrow{j^\infty_\Rho }H^{\bullet}_{\Rho_F}(\R^3,\pi_\tau)
\end{equation*}
is surjective by the description of the non-trivial classes in \cref{thm:open_book_formal}. Hence the long exact sequence \eqref{eq: les R} induces short exact sequences
\begin{equation*}
    0\to H^{\bullet}_{\Rho_f}(\R^3,\pi_\tau)\to H^{\bullet}(\R^3,\pi_\tau)\to H^{\bullet}_{\Rho_F}(\R^3,\pi_\tau)\to 0
\end{equation*}
and the result follows from \cref{thm:open_book_formal}, \cref{prop: formal along y and xy} and \eqref{eq: ismorphism xy-flat z-flat}.
\end{proof}

\section{The Lie algebra of semi open book-type}\label{sec:typeiv}
In this section we treat the Lie algebra given by the relations
\[ [e_1,e_3]=e_1\quad \text{ and }\quad [e_2,e_3]=e_1+e_2.\]
The corresponding linear Poisson structure on $\R^3$ is given by
\begin{align*}
    \pi= (E+x\partial_y)\wedge \partial_z
\end{align*}
To compute the Poisson cohomology of $\pi$ we make use of the long exact sequence in \eqref{eq: les R} for $\pi$.

\subsection{Formal Poisson cohomology}

In this section we compute the formal Poisson cohomology groups $H^{\bullet}_{\Rho_F}(\R^3,\pi)$. The groups are described in the following proposition.

\begin{proposition}\label{prop: type iv}
The formal Poisson cohomology groups $H^{\bullet}_{\Rho_F}(\R^3,\pi)$ are uniquely characterized by the following representatives:
\begin{itemize}
    \item in degree $0$:
    \[ c\in \R; \quad \qquad\]
    \item in degree $1$:
    \[ c_1 x\partial_y+c_2\partial_z \qquad \text{ for } \ c_1,c_2\in \R;\]
    \item in degree $2$:
    \[ cy\partial_{x}\wedge\partial_z \qquad \text{ for } \ c\in \R; \]
    \item and the third Poisson cohomology group is trivial.
\end{itemize}
Moreover, the wedge product and the Schouten-Nijenhuis bracket are trivial in cohomology.
\end{proposition}
\begin{proof}
\underline{In degree 0 :} \ We have for $g\in\mathfrak{X}_{\Rho_F}^0(\R^3)$ that
\begin{align*}
    \diff  _{\pi}g= (E+x\partial_y)g\partial_z- (\partial_z g)(x\partial_x+(x+y)\partial_y)
\end{align*}
If we write $g=\sum_{i,j}g_{ij}(z)x^iy^j$ we obtain for cocycles the equation
\begin{align*}
    \partial_zg_{ij}(z)=0 \quad \text{ and }\quad (i+j)g_{ij}(z)+(j+1)g_{i-1j+1}(z)=0.
\end{align*}
for all $i,j\ge 0$. The left equations imply the all $g_{ij}$'s are constants. From the right equations we get $g_{0j}=0$ for $j>0$ and then inductively for $i\to i+1$ and fixed $n=i+j$ the result.

\noindent{\underline{In degree 1 :}} For $X \in\mathfrak{X}_{\Rho_F}^1(\R^3)$ we obtain from \eqref{eq: general 1x}, \eqref{eq: general 1y} and \eqref{eq: general 1z}:
\begin{align*}
    W^x(\pi)=&\,(E+x\partial_y-1)X^y-X^x+(x+y)\partial_zX^z\\
    W^y(\pi)= -(E +x\partial_y-1)X^x-x\partial_zX^z &\, \qquad \text{ and }\qquad 
    W^z(\pi)= x\partial_zX^y-(x+y)\partial_zX^x
\end{align*}
Let $X^k=\sum_{i,j}X_{ij}^k(z)x^iy^j$. Following our general idea to compute formal Poisson cohomology and using the forumla for the differential in degree $0$, we may, by a choice of a coboundary, assume that
\begin{align}\label{eq: assumption 41}
    X^z\in C^{\infty}(\R) \quad \text{ and }\quad X_{10}^x=0
\end{align}
Setting $W^x(\pi)$ and $W^y(\pi)$ equal to zero implies for $i+j\ne 1$ the equations:
\begin{align*}
    X^x_{ij}(z)=&\, (i+j-1)X^y_{ij}(z)+(j+1)X^y_{i-1j+1}(z)\\
     0=&\, (i+j-1)X^x_{ij}(z)+(j+1)X^x_{i-1j+1}(z)
\end{align*}
Hence we obtain again $X^x_{0j}=X^y_{0j}=0$ for $j\ne 1$ and then inductively for $i\to i+1$ and fixed $i+j=n\ne 1$ that
\begin{align*}
    X^x_{ij}(z)=X^y_{ij}(z)=0.
\end{align*}
For $(i,j)\in \{(1,0),(0,1)\}$ we obtain the equations
\begin{align*}
    -\partial_z X^z_{00}(z)=X^y_{01}(z),\qquad X^x_{01}(z)=\partial_z X^z_{00}(z)\qquad \text{ and }\qquad -\partial_z X^z_{00}(z)=X^x_{01}(z)
\end{align*}
which imply that
\begin{align*}
    X^x_{01}(z)=\partial_z X^z_{00}(z)=X^y_{01}(z)=0
\end{align*}
Finally the equation for $W^z(\pi)=0$ implies that
\begin{align*}
    \partial_z X^y_{10}(z)=&0.
\end{align*}
which proves the result in this degree.

\noindent{\underline{In degree 2:}} For $W \in\mathfrak{X}_{\Rho_F}^2(\R^3)$ we obtain for the differential by \eqref{eq: general 2} the expression
\begin{align*}
    \diff  _{\pi}W=(x\partial_zW^x+(x+y)\partial_zW^y+(2-E-x\partial_y)W^z)\partial_x\wedge \partial_y\wedge\partial_z
\end{align*}
Let $W^k=\sum_{ij}W_{ij}^kx^iy^j$. From the equations in degree $1$ we obtain for a coboundary the equation
\begin{align*}
    W^x_{ij}(z)=&(i+j-1)X^y_{ij}(z)+(j+1)X^y_{i-1j+1}(z) -X^x_{ij}(z)+\partial_z X^z_{i-1j}(z)+\partial_z X^z_{ij-1}(z)\\ 
    -W^y_{ij}(z)=&(i+j-1)X^x_{ij}(z)+(j+1)X^x_{i-1j+1}(z)+\partial_z X^z_{i-1j}(z)
\end{align*}
We showed in degree $1$ that for fixed $i+j=n\ne 1$ this linear maps are injective for $X^z_{ij}(z)=0$. Considering only $X^x_{ij},X^y_{ij}, W^x_{ij},W^y_{ij}\in \R$ these are linear maps between finite dimensional vector spaces  of the same dimensions we obtain surjectivity for every $n$. Since the maps are $C^{\infty}(\R)$ linear in these components we may assume that
\begin{align*}
    W^x_{ij}(z)=W^y_{ij}(z)=0 \qquad \text{ for } \ i+j=n\ne 1.
\end{align*}
by choosing $X^x_{ij}$ and $X^y_{ij}$ accordingly. For $(i,j)=(1,0)$ we get the equations
\begin{equation}\label{eq: algebraic 4}
    \begin{aligned}
         W^x_{10}(z)=X^y_{01}(z) -X^x_{10}(z)+\partial_z X^z_{00}(z),&\,  \qquad W^x_{01}(z)= -X^x_{01}(z)+\partial_z X^z_{00}(z) \\
        -W^y_{10}(z)=&\, X^x_{01}(z)+\partial_z X^z_{00}(z)
    \end{aligned}
\end{equation}
Moreover, for $(i,j)\in \{(2,0),(0,2)\}$ we get for $W^z$ the equations
\begin{align*}
    W^z_{20}(z)= \partial_z X^y_{10}(z)-\partial_z X^x_{10}(z)\qquad \text{ and }\qquad 
    W^z_{02}(z)= -\partial_z X^x_{01}(z)
\end{align*}
Hence we may assume that
\begin{align*}
    W^x_{ij}(z)=W^y_{ij}(z)=0 \qquad \text{ for } \ (i,j)\ne (0,1), \qquad W^x_{01}(z)=0\qquad  \text{ and }\qquad W^z_{20}(z)=W^z_{02}(z)=0. 
\end{align*}
Under this assumption the cocycle condition for $W$ is, for $i+j=n\ne 2$ equivalent to the equation
\begin{align}\label{eq: 42g}
    0=(i+j-2)W^z_{ij}(z)+(j+1)W^z_{i-1j+1}(z)
\end{align}
and for $(i,j)\in \{(2,0),(1,1),(0,2)\}$ we obtain the equations
\begin{align}\label{eq: 422}
    0=W^z_{11}(z), \qquad 0=\partial_zW^x_{01}(z)+\partial_zW^y_{01}(z)\qquad \text{ and }\qquad
    \partial_zW^y_{01}(z)=0
\end{align}
and hence the result follows.

\noindent{\underline{In degree 3:}} The equations \eqref{eq: 42g} and \eqref{eq: 422} imply the result, as they imply that the linear maps are injective in the given elements. Hence surjectivity follows with a similar argument as in degree $2$ for fixed degree $n=i+j\ne 2$ and for $n=2$ we can read it off of \eqref{eq: 422}.

\noindent{\underline{Algebraic structure:}} The triviality follows from a computation on the given representatives and \eqref{eq: algebraic 4}.
\end{proof} 

\subsection{Flat Poisson cohomology}
For the flat cohomology we use the diffeomorphism
\begin{equation*}
    \begin{array}{cccc}
         \varphi :& (0,\infty)\times S^1\times \R &\to &\R^3\backslash \Rho \\
         &(r,\theta,s)&\mapsto & (r\cos(\theta), r(\log(r)\cos(\theta)+\sin(\theta)),s) 
    \end{array}
\end{equation*}
We note that the differential is given by
\begin{align*}
    \diff   \varphi (r,\theta,s)= \begin{pmatrix}
    \cos(\theta)&-r\sin (\theta)&0\\
    (1+\log(r))\cos(\theta)+\sin(\theta)&r(\cos(\theta)-\log(r)\sin(\theta))&0\\
    0&0&1
    \end{pmatrix} 
\end{align*}
which is an isomorphism since
\begin{align*}
    \det(\diff   \varphi)= r(1+\cos(\theta)\sin(\theta))>0.
\end{align*}
It's inverse is given by
\begin{align*}
    \diff   \varphi ^{-1} (\varphi(r,\theta,s))= \begin{pmatrix}    \frac{\cos(\theta)-\log(r)\sin(\theta)}{1+\cos(\theta)\sin(\theta)}&\frac{\sin (\theta)}{1+\cos(\theta)\sin(\theta)}&0\\
    -\frac{(1+\log(r))\cos(\theta)-\sin(\theta)}{r(1+\cos(\theta)\sin(\theta))}&\frac{\cos(\theta)}{r(1+\cos(\theta)\sin(\theta))}&0\\0&0&1
    \end{pmatrix}
\end{align*}
Similar as \cref{proposition: positive tau isomorphism} we obtain the following proposition.
\begin{proposition}
The map $\varphi$ induces an isomorphism of rngs:
\begin{align*}
    \varphi^*: C^{\infty}_{\Rho_f}(\R^3)\diffto C^{\infty}_{R_f}((0,\infty)\times S^1\times \R).
\end{align*}
\end{proposition}
Note that $\varphi$ is a Poisson diffeomorphism where it is well-defined since
\begin{align*}
    \varphi_*(r\partial_r\wedge \partial_s)=(E+x\partial_y)\wedge \partial_z
\end{align*}
Hence $\varphi$ induces an isomorphism in flat Poisson cohomology. As a consequence of Proposition \ref{prop:flat_cylindrical} we get
\begin{corollary}
The flat Poisson cohomology groups $H^{\bullet}_{\Rho_f}(\R^3,\pi)$ are trivial in all degrees.
\end{corollary}
We are now ready to prove \cref{thm: type iv}.
\begin{proof}[Proof of \cref{thm: type iv}]
Since the cohomology groups $H^{\bullet}_{\Rho_f}(\R^3,\pi)$ are trivial, \eqref{eq: les R} induces isomorphisms
\begin{equation*}
    H^{\bullet}(\R^3,\pi)\diffto H^{\bullet}_{\Rho_F}(\R^3,\pi).
\end{equation*}
Therefore the result follows from \cref{prop: type iv}.
\end{proof}
\section{The Lie algebra of spiral-type}\label{sec:typevii}
In this section we want to compute the Poisson cohomology of the linear Poisson structure associated with the Lie algebra given by
\begin{align*}
    [e_1,e_3]= \tau e_1- e_2 \quad \text{ and }\quad [e_2,e_3]= e_1 +\tau e_2.
\end{align*}
for $0<\tau$. The corresponding linear Poisson structure is given by
\begin{align*}
    \pi_{\tau} = (\tau E+ T)\wedge \partial_z
\end{align*}
For the computation of the Poisson cohomology of $(\R^3,\pi_{\tau})$ we use once more the long exact sequence \eqref{eq: les R}.

\subsection{Formal Poisson cohomology}
Denote by $\mathfrak{X}_{\Rho_F}^{\bullet}(\R^3)$ the multivector fields with coefficients in the ring $C^{\infty}(\R)[[x,y]]$ and denote by $H^{\bullet}_{\Rho_F}(\R^3,\pi_{\tau})$ the associated Poisson cohomology.

\begin{proposition}\label{proposition: formal 7}
For $0<\tau$, the formal Poisson cohomology groups $H^{\bullet}_{\Rho_F}(\R^3,\pi_{\tau})$ are uniquely characterized by the following representatives in the different degrees:
\begin{itemize}
    \item in degree $0$:
    \[ c\in \R; \]
    \item in degree $1$:
    \[ c_1E+ c_2\partial_z\qquad \text{ for } c_1,c_2\in \R;\]
    \item in degree $2$:
    \[ cE\wedge\partial_z\qquad \text{ for } c\in \R;\]
    \item and the third formal Poisson cohomology group is trivial.
\end{itemize}
\end{proposition}
For the proof we will use the following Lemma.
\begin{lemma}\label{lemma: nice basis}
Every vector field $X\in \mathfrak{X}_{\Rho_F}$ can uniquely be written as
\begin{align*}
    X^E E+X^{T}T+X^x\partial_x + X^y\partial_y +X^z\partial_z
\end{align*}
where $X^E,X^T,X^z\in C^{\infty}(\R)[[x,y]]$ and $X^x,X^y\in C^{\infty}(\R)[[y]]$. Similarly, every bivector field $W\in \mathfrak{X}^2_{\Rho_F}$ can uniquely be written as
\begin{align*}
    W^T E\wedge \partial_z +W^{E}T\wedge\partial z+W^x\partial_x\wedge\partial_z +W^y\partial_y\wedge \partial z +W^z\partial_x\wedge \partial_y
\end{align*}
where $W^E,W^T,W^z\in C^{\infty}(\R)[[x,y]]$ and $W^x,W^y\in C^{\infty}(\R)[[y]]$.
\end{lemma}
\begin{proof}
This follows from the simple fact that $\R[[x,y]]$ decomposes as
\begin{align*}
    \R[[x,y]]= x\cdot \R[[x,y]] \oplus \R[[y]]
\end{align*}
\end{proof}
\begin{proof}[Proof of \cref{proposition: formal 7}]
\underline{In degree $0$:} In degree zero we have for $g\in\mathfrak{X}_{\Rho_F}^0(\R^3)$ that
\begin{align}\label{eq: 70}
    \diff  _{\pi_{\tau}}g= (\tau E+ T)g\partial_z- (\partial_z g)(\tau E+ T)
\end{align}
If we write $g=\sum_{i,j}g_{ij}(z)x^iy^j$ we obtain for cocycles the equations
\begin{align*}
    \tau \partial_zg_{i-1j}(z)- \partial_zg_{ij-1}(z)=0, \qquad \text{ and }\qquad
    \tau (i+j)g_{ij}(z)+ (j+1)g_{i-1j+1}(z)-(i+1)g_{i+1j-1}(z)=0.
\end{align*}
The second equations imply for fixed $i+j=n\ne 0$ and $i\to i+1\le n$ that
\begin{align*}
    g_{i+1n-i-1}(z)=c_ig_{0n}(z)
\end{align*}
for a positive constants $c_i >0$. For $i=n$ we get the equation
\begin{align*}
    g_{1n-1}=-cg_{0n}
\end{align*}
However, for $(i,j)=(0,n)$ we obtain the equation
\begin{align*}
    \tau ng_{0n}(z)=- g_{1n-1}(z).
\end{align*}
which implies that
\begin{align*}
    g_{ij}(z)=0 \qquad \text{ for } i+j\ne 0
\end{align*}
Finally the first equation implies that $g\in \R$.

\noindent{\underline{In degree $1$:}} For $X \in\mathfrak{X}_{\Rho_F}^1(\R^3)$ let us write $X$ as in Lemma \ref{lemma: nice basis} and let us denote its differential by
\begin{align*}
    \diff   _{\pi_{\tau}}X=& W^T(\pi) E\wedge \partial_z +W^{E}(\pi)T\wedge\partial z +W^x(\pi)\partial_y\wedge \partial_z +W^y(\pi)\partial_z\wedge\partial_x +W^z(\pi)\partial_x\wedge \partial_y
\end{align*}
The components $W^i(\pi)$ are given by
\begin{align*}
    W^T(\pi)=&\, (\tau E+T)X^E+\partial_y X^x+\tau \partial_z X^z, \qquad 
    W^E(\pi)= (\tau E+T)X^T+\partial_y X^y+\partial_z X^z \\
    W^x(\pi)=&\, \tau (y\partial_y -1) X^y-(1+y\partial_y)X^x\qquad -W^y(\pi) = \tau (y\partial_y -1) X^x+(1+y\partial_y)X^y,\quad  \qquad \\
    W^z(\pi)=&\, -r^2 \partial_zX^E+\tau r^2\partial_zX^T+ (\tau x-y)\partial_zX^y -(\tau y +x)\partial_z X^x
\end{align*}
where $r^2:=x^2+y^2$. Since we are interested in cohomology, we can use the same argument as in degree zero to show that we may choose a coboundary such that we may assume that
\begin{align*}
    X^z_{ij}(z)=0 \quad \text{ for } \ (i,j)\ne (0,0)  \qquad \text{ and }\qquad X^T_{00}(z)=0.
\end{align*}
Assuming that $X$ is a cocycle, the equations for $W^x(\pi)$ and $W^y(\pi)$ then read as:
\begin{align}
     0= \tau (j-1)X^y_j-(1+j)X^x_j\qquad \text{ and }\qquad 0=\tau (j-1)X^x_j+(1+j)X^y_j \label{eq: 714}
\end{align}
for all $j$, implying that
\begin{align*}
    X^x_{j}(z)= X^y_{j}(z)=0 \qquad \text{ for all }\ j.
\end{align*}
The equations for $W^E$ and $W^T$ then become
\begin{align}
     0=&\,  \tau(i+j)X^E_{ij}(z)-(i+1)X^E_{i+1j-1}(z)+(j+1)X^E_{i-1j+1}(z)+\tau \delta_{i0}\delta_{j0}\partial_z X^z_{00}(z) \label{eq: 711}\\
    0=&\, \tau(i+j)X^T_{ij}(z)-(i+1)X^T_{i+1j-1}(z)+(j+1)X^T_{i-1j+1}(z)+\delta_{i0}\delta_{j0}\partial_z X^z_{00}(z)\nonumber
\end{align}
These equations imply inductively for fixed $0\ne n=i+j $ and $i\to i+1$ with the same argument as in degree $0$ that
\begin{align*}
    X^E_{ij}(z)=0 \qquad \text{ for }\ (i,j)\ne (0,0)\qquad \text{ and }\qquad X^T_{ij}(z)=0 \qquad \text{ for all }\ (i,j).
\end{align*}
Equation \eqref{eq: 711} implies for $(i,j)=0$ that $X^z\in \R$ and finally, the equation for $W^z(\pi)=0$ implies that $X^E\in \R$.

\noindent{\underline{In degree $2$:}} Let us write $W \in\mathfrak{X}_{\Rho_F}^2(\R^3)$ as in Lemma \ref{lemma: nice basis} and note that by the arguments in degree $1$ we may assume that 
\begin{align*}
    W^x=W^y=W^E=0, \qquad  W^T\in C^{\infty}(\R), \qquad \text{ and }\qquad W^z\in C^{\infty}(\R)[[x,y]]\backslash x^2C^{\infty}(\R).
\end{align*}
We obtain for the differential the expression
\begin{align*}
    \diff  _{\pi_{\tau}}W= (-r^2\partial_zW^T+(\tau(2-E) -T)W^z)\partial_x\wedge \partial_y\wedge\partial_z
\end{align*}
The cocycle condition then can for $i+j\ne 2$ be written as
\begin{align}\label{eq: 72}
    0=&\,\tau(2-i-j)W^z_{ij}(z)+(i+1)W^z_{i+1j-1}(z)-(j+1)W^z_{i-1j+1}(z)
\end{align}
Hence arguing as in degree $0$ for $n=i+j\ne 2$ fixed and $j\to j+1$ we obtain 
\begin{align*}
    W^z_{ij}(x)=0 \qquad \text{ for } i+j\ne 2
\end{align*}
For $i+j=2$ we obtain the three equations
\begin{align}\label{eq: 722}
    0=&\,\partial_zW^T_{00}(z)-W^z_{11}(z),\qquad
    0=-2W^z_{02}(z)\qquad \text{ and }\qquad
    0=\partial_zW^T_{00}(z)+W^z_{11}(z)
\end{align}
and hence we obtain that
\begin{align*}
    W^z_{ij}=0 \qquad \text{ for all } i,j \qquad \text{ and } \qquad W^T\in \R.
\end{align*}

\noindent{\underline{Degree 3:}} We can derive from equations \eqref{eq: 72} and \eqref{eq: 722} that the third cohomology group is trivial.
\end{proof}

\subsection{The flat Poisson complex}
To compute the flat cohomology we use the fact that, away from the $z$-axis the Poisson structure induces a codimension one symplectic foliation.

By removing the $z$-axis from the Poisson manifold $(\R^3,\pi_{\tau})$, we obtain a codimension one symplectic foliation $(\F_{\tau},\omega_{\tau})$ which is unimodular with defining one-form $\varphi_{\tau}$ given by
\begin{align*}
    \varphi_{\tau} = \frac{(x+\tau y)\diff x +(y-\tau x)\diff y}{x^2+y^2}= \frac{\diff r}{r}-\tau \diff \theta  
\end{align*}
Therefore, the techniques from \cref{subsec:corank1} can be used to describe its cohomology. Consider on $\R^3\backslash\{x=y=0\}$ the vector field 
 \begin{align*}
   V_{\tau}:=&\frac{1}{2}\left(\left(x+\frac{y}{\tau}\right)\partial_x+\left(y-\frac{x}{\tau}\right)\partial_y\right)= \frac{1}{2}\left(E-\frac{1}{\tau}T\right),
  \end{align*}
and note that $\varphi_{\tau} (V_{\tau})=1$. The unique extension $\widetilde{\omega}_{\tau}$ of the leafwise symplectic structure satisfying $i_{V_{\tau}}\widetilde{\omega}_{\tau}=0$ is given by 
  \begin{align*}
   \widetilde{\omega}:=-\frac{1}{2(x^2+ y^2)}\left(\left(\frac{x}{\tau}-y\right)\diff x +\left(x+\frac{y}{\tau}\right)\diff y\right)\wedge \diff z 
  \end{align*}
Therefore, the Poisson complex of $(\R^3\backslash\{x=y=0\},\pi)$ fits into the short exact sequence \eqref{ses poisson unimod}, with $\varphi=\varphi_{\tau}$. However, since the singularities of $\varphi_{\tau}$ and $\widetilde{\omega}_{\tau}$ are of finite order, we can apply the same reasoning and obtain a similar short exact sequence for the $z$-flat Poisson cohomology. Denote by $\Omega_{\Rho_f}^{\bullet}(\R^3)$ the space of differential forms on $\R^3$ which are flat at the $z$-axis. The following holds:

\begin{proposition}\label{ses for flat things}
The $z$-flat Poisson complex of $(R^3,\pi_{\tau})$ fits into the short exact sequence:
 \[ 0\to (\varphi_{\tau}\wedge \Omega_{\Rho_f}^{\bullet}(\R^3),\diff )\xrightarrow{j_{\varphi_{\tau}}} (\mathfrak{X}^{\bullet}_{\Rho_f}(\R^3),\diff _{\pi})\xrightarrow{p_{\varphi_{\tau}}} (\varphi_{\tau}\wedge \Omega^{\bullet-1}_{\Rho_f}(\R^3),\diff )\to 0,
 \]
 where $j_{\varphi_{\tau}}=(-\pi_{\tau}^{\sharp})\circ i_{V_{\tau}}$ and  $p_{\varphi_{\tau}}=e_{\varphi_{\tau}}\circ (-\widetilde{\omega}^{\flat}_{\tau})\circ i_{\varphi_{\tau}}$.
\end{proposition}
\begin{proof}
First, note that $\varphi_{\tau}\wedge \Omega^{\bullet}_{\Rho_f}(\R^3)$ is indeed well-defined: if $\eta$ is a flat form at $0$, then $\varphi_{\tau}\wedge \eta$ extends smoothly at zero and is also flat. The same applies also to the maps $\widetilde{\omega}^{\flat}_{\tau}$ and $p_{\varphi_{\tau}}$, hence the maps $j_{\varphi_{\tau}}$ and $p_{\varphi_{\tau}}$ are well-defined. They are chain maps because they satisfy this condition away from the $z$-axis. In order to show that the sequence is exact, note that also the maps $p_{V_{\tau}}$ and $j_{V_{\tau}}$ defined in \eqref{dual_maps} induce maps on flat forms/multi-vector fields. Relations \eqref{dual_maps_relations} (which still hold, because they hold away from the $z$-axis) imply that the sequence in the statement is indeed exact. 
\end{proof}

We call the cohomology of the complex 
\begin{equation*}
(\varphi_{\tau}\wedge\Omega^{\bullet}_{\Rho_f}(\R^3),\diff )
\end{equation*}
the flat foliated cohomology which we denote by $H^{\bullet}_{\Rho_f}(\mathcal{F})$. We have the following result:
\begin{proposition}\label{flat foliated cohomology}
    The flat foliated cohomology $H^{\bullet}_{\Rho_f}(\mathcal{F})$ is trivial in all degrees. 
\end{proposition}
As an immediate consequence we obtain for the $z$-flat Poisson cohomology:
\begin{corollary}
The $z$-flat Poisson cohomology groups $H^{\bullet}_{\Rho_f}(\R^3,\pi_{\tau})$ are trivial in all degrees.
\end{corollary}
\begin{proof}[Proof of \cref{flat foliated cohomology}]    
In order to calculate the cohomology of $(\varphi_{\tau}\wedge\Omega^{\bullet}_{\Rho_f}(\R^3),\diff )$, we use a retraction onto the $z$-axis along the leaves of the foliation. Define the retraction as follows:
\begin{align*}
p (x,y,z)&:=(0,0,z)
\end{align*}
Note that $p$ induces via pullback the zero map on the $z$-flat foliated complex, i.e. 
\[p^*=0:(\varphi_{\tau}\wedge\Omega_{\Rho_f}^{\bullet}(\R^3),\diff )\to (\varphi_{\tau}\wedge\Omega_{\Rho_f}^{\bullet}(\R^3),\diff ).\]
We show that $p^*$ is a quasi-isomorphism, i.e.\, an isomorphism in cohomology. For the proff we use the flow of the vector field:
  \begin{align*}
   W_{\tau}=& -\tau E -T =\pi_{\tau}^{\sharp}(\diff z).
  \end{align*}
Note that $W_{\tau}$ is tangent to $\F_{\tau}$. Its flow is given by:
  \begin{equation*}
  \begin{array}{cccc}
   \phi^{\tau}_t:&[0,\infty)\times \R^3 &\to&\R^3\\
   &(t,x,y,z)&\mapsto&(x_t^{\tau},y_t^{\tau},z_t^{\tau})
  \end{array}
  \end{equation*}
  where 
  \begin{align*}
      x_t^{\tau}:=&\, e^{-t\tau}(x\cos(2\pi t)- y\sin(2\pi t))\\
      y_t^{\tau}:=&\, e^{-t\tau}(y\cos(2\pi t)+x\sin(2\pi t))\\
      z_t^{\tau}:=&\, z
  \end{align*}
From the properties of the vector field $W_{\tau}$ we immediately obtain that $\phi^{\tau}_t$ fixes the $z$-axis and preserves the complex $\varphi_{\tau}\wedge\Omega^{\bullet}_{\Rho_f}(\R^3)$. 
By a direct calculation, we have that 
\begin{equation}\label{eq: at time t}
\left(\phi_t^\tau\right)^*(\alpha)-\alpha=\diff  \circ h_t^\tau(\alpha)+h_t^\tau\circ \diff (\alpha),    
\end{equation}
where
\[h_t^{\tau}:\varphi_\tau \wedge\Omega^{\bullet}_{\Rho_f}(\R^3)\to \varphi_\tau\wedge\Omega^{\bullet-1}_{\Rho_f}(\R^3), \ \ h_t^{\tau}(\alpha):=\int_0^{t}i_{W_{\tau}}(\phi_s^{\tau})^*(\alpha)\diff  s.\]
The explicit formulas for $\phi_t^\tau$ imply that: 
\begin{align}\label{infinity now}
\lim_{t\to\infty}\phi_t^\tau=p, 
\end{align}
with respect to the compact-open $C^{\infty}$-topology, and for any $\alpha\in \Omega_{\Rho_f}^{\bullet}(\R^3)$, the limit 
   \begin{align}\label{limit h}
    h^\tau(\alpha):=\lim_{t \to \infty}{h_{t}^\tau(\alpha)}\in \Omega_{\Rho_f}^{\bullet-1}(\R^3),
   \end{align}
exists with respect to the compact-open $C^{\infty}$-topology. Recall that the existence of a limit with respect to the compact-open $C^{\infty}$-topology means that all partial derivatives converge uniformly on compact subsets (see e.g. ~\autocite[section 7.2]{Marcut2019}). Since the limit \eqref{limit h} is uniform on compact subsets with respect to all $C^k$-topologies, $h$ satisfies
   \begin{align*}
    \diff  h^\tau(\alpha)&=\lim_{t \to \infty}{\diff  h_{t}^\tau(\alpha)}.
   \end{align*}
From \eqref{eq: at time t} and \eqref{infinity now}, we obtain that for any $\alpha\in \Omega_{\Rho_f}^{\bullet}(\R^3)$
\[0=p^*(\alpha)=\lim_{t\to\infty}(\phi_t^\tau)^*(\alpha)=\alpha+\diff \circ h^\tau(\alpha)+h^\tau\circ \diff  (\alpha)\]
holds for the compact-open $C^{\infty}$-topology. 

Finally, since $h_t^\tau$ commutes with $e_{\varphi_{\tau}}$, and this condition is closed, we have that 
\[ h^\tau(\varphi_\tau\wedge\Omega^{\bullet}_{\Rho_f}(\R^3))\subset \varphi_\tau\wedge\Omega^{\bullet-1}_{\Rho_f}(\R^3).\] 
Thus, the above relation holds on $\varphi_\tau\wedge\Omega^{\bullet}_{\Rho_f}(\R^3)$, and so we obtain also \cref{flat foliated cohomology}.
\end{proof}

\printbibliography

\end{document}